\documentclass{amsart}
\usepackage{amsfonts}
\usepackage{amsmath}
\usepackage{amssymb}
\usepackage{textcomp}
\usepackage{tikz}
\usepackage{enumerate}
\usepackage[arrow, matrix, curve, graph]{xy}
\usetikzlibrary{matrix,arrows,decorations.pathmorphing}
\usetikzlibrary{decorations.markings}
\usetikzlibrary{decorations.markings}
\newtheorem{theorem}{Theorem}[section]

\newtheorem{lemma}[theorem]{Lemma}

\newtheorem{corollary}[theorem]{Corollary}

\theoremstyle{definition}

\newtheorem{definition}[theorem]{Definition}
\newtheorem{notation}[theorem]{Notation}

\newtheorem{examples}[theorem]{Examples}

\newtheorem{remark}[theorem]{Remark}

\theoremstyle{plain}
\newtheorem *{Main Theorem}{Main Theorem}
\newtheorem *{Corollary B}{Corollary B}
\newtheorem *{Corollary E}{Corollary E}
\newtheorem *{Theorem B}{Theorem B}
\newtheorem *{Projective Schur's Lemma}{Projective Schur's Lemma}
\newtheorem *{Graded Artin Wedderburn Theorem}{Graded Artin Wedderburn Theorem}
\newtheorem *{Higgs' Conjecture}{Higgs' Conjecture}
\newtheorem *{Theorem C}{Theorem C}
\newtheorem *{Theorem D}{Theorem D}
\newtheorem *{BSZ}{Generalized BSZ Theorem}
\newtheorem *{Theorem E}{Theorem E}
\newtheorem *{Theorem F}{Theorem F}
\newtheorem *{Theorem G}{Theorem G}
\newtheorem *{Theorem H}{Theorem H}
\newtheorem *{Corollary I}{Corollary I}
\newtheorem *{Theorem J}{Theorem J}
\newtheorem *{Corollary K}{Corollary K}
\newtheorem *{Theorem L}{Theorem L}
\newtheorem *{Question1}{Question 1}
\newtheorem *{Question2}{Question 2}
\newtheorem *{Problem 1}{Problem 1}
\newtheorem *{Problem 2}{Problem 2}
\newtheorem *{Problem 3}{Problem 3}

\newcommand{\Id}{\text {Id}}

\newcommand{\p}{\text{Pic}_k}
\newcommand{\e}{\text{eq}}
\newcommand{\au}{\text{Aut}_k(K)}

\newcommand{\Z}{{\mathbb Z}}

\newcommand{\CGK}{{\mathcal{C}_k{(G,K)}}}

\newcommand{\CPK}{{\mathcal{C}_k{(\phi,K)}}}

\newcommand{\CK}{{\mathcal{C}_k{(\phi)}}}

\newcommand{\EX}{{\text{Ext}_k{(\phi)}}}
\newcommand{\EXCS}{{\text{Ext}_k{(\phi,\mathcal{CS})}}}
\newcommand{\CO}{{\mathcal{C}^{\text{Out}}_k{(\phi)}}}
\newcommand{\COCS}{{\mathcal{C}^{\text{Out}}_k{(\phi,\mathcal{CS})}}}

\newcommand{\CS}{{\mathcal{CS}(K)}}

\newcommand{\CSP}{{\mathcal{CS}^{\phi}(K)}}
\newcommand{\BrK}{{\text{Br}(K)}}
\newcommand{\Brk}{{\text{Br}(k)}}
\newcommand{\Br}{{\text{Br}}}
\newcommand{\Aut}{{\text{Aut}}}
\newcommand{\Inn}{{\text{Inn}}}
\newcommand{\Out}{{\text{Out}}}
\newcommand{\Hom}{{\text{Hom}}}
\newcommand{\res}{{\text{res}}}

 \newcommand{\bigslant}[2]{{\raisebox{.1em}{$#1$}\left/\raisebox{-.1em}{$#2$}\right.}}
\newcommand{\veq}{\mathrel{\rotatebox{90}{$=$}}}

\numberwithin{equation}{section}

\tikzset{node distance=3cm, auto}

\begin{document}
\title
[Realization-obstruction of Clifford system extensions]
{Realization-obstruction exact sequences for Clifford system extensions}
\author{Yuval Ginosar}

\address{Department of Mathematics, University of Haifa, Haifa 3498838, Israel}
\email{ginosar@math.haifa.ac.il}

\begin{abstract}
For every action $\phi\in\text{Hom}(G,\au)$ of a group $G$ on a commutative ring $K$ we introduce two abelian monoids.
The monoid $\text{Cliff}_k(\phi)$ consists of equivalence classes of strongly $G$-graded algebras of type $\phi$
up to $G$-graded Clifford system extensions of $K$-central algebras.
The monoid $\CK$ consists of equivariance classes of homomorphisms of type $\phi$ from $G$ to the Picard groups of $K$-central algebras (generalized collective characters).
Furthermore, for every such $\phi$ there is an exact sequence of abelian monoids
$$0\to H^2(G,K^*_{\phi})\to\text{Cliff}_k(\phi)\to\CK\to H^3(G,K^*_{\phi}).$$
This sequence describes the obstruction to realizing a generalized collective character of type $\phi$, that is it determines if such a character is associated to some strongly $G$-graded $k$-algebra.
The rightmost homomorphism is often surjective, terminating the above sequence.
When $\phi$ is a Galois action, then the well-known restriction-obstruction sequence of Brauer groups is an image of an exact sequence of sub-monoids appearing in the above sequence.
\end{abstract}

\date{\today}

\maketitle

\bibliographystyle{abbrv}
\tikzset{node distance=3cm, auto}

\section{Introduction}
In their paper \cite{CG01}, A. M. Cegarra and A. R. Garz\'on define a Generalized Collective Character (GCC) to be a homomorphism from a group $G$ to the Picard group of an algebra $R$.
This indeed generalizes the notion of a collective character, where the image of $G$ lies in the outer automorphisms of $R$, naturally embedded in the Picard group of $R$.
A GCC $\Phi$ determines an action $\phi$ of $G$ on the center $\mathcal{Z}(R)$ of $R$. In particular, the central units $\mathcal{Z}(R)^*$ are endowed with a $G$-module structure.
We say that $\phi$ is the {\it associated action} of $\Phi$. Alternatively, we say that $\Phi$ is {\it of type} $\phi\in$Hom$(G,$Aut$(\mathcal{Z}(R)))$.

On the other hand, every strongly $G$-graded algebra $A$ with base algebra $R$ determines a GCC $\Phi=\Phi(A):G\to$Pic$(R)$, which we term the {\it associated GCC} of $A$.
We say that such a strongly $G$-graded algebra is {\it of type} $\phi\in$Hom$(G,$Aut$(\mathcal{Z}(R)))$ if its associated GCC is of this type.
Further, a GCC $\Phi$ gives rise to a well-defined class in the third cohomology of $G$ with coefficients in the $G$-module $\mathcal{Z}(R)^*$ determined by $\phi$.
This cohomology class $T(\Phi)$ is an obstruction to realizing $\Phi$ as the GCC associated to some strongly $G$-graded algebra with base algebra $R$.
Once the obstruction vanishes, then such $G$-graded algebras with associated GCC $\Phi$ are in one-to-one correspondence, up to equivalence of
Clifford system extensions (see Definition \ref{cliffeq}), with the second cohomology of $G$ over the same coefficients.
More precisely, $H^2(G,\mathcal{Z}(R)^*_{\phi})$ acts freely and transitively on the classes of
Clifford system extensions with associated GCC $\Phi:G\to$Pic$(R)$ of type $\phi\in$Hom$(G,$Aut$(\mathcal{Z}(R)))$.

The above setup hints that there is a connecting exact sequence behind. In a following paper \cite{CG03}, Cegarra and Garz\'on do construct such exact sequences in the
more general setup of $\Gamma$-groups in two cases,
both for commutative base rings $R$.
The first \cite[Theorem 5.7]{CG03} is for central graded algebras, i.e. with a trivial associated action, whereas the second \cite[Theorem 5.9]{CG03} admits an associated Galois action.

The objective of this paper is to view the above maps $A\mapsto \Phi(A)$ and $\Phi\mapsto T(\Phi)$ as certain algebraic morphisms, and then to tie them up in a family of exact sequences.

Here is a brief outline.
Let $K$ be a commutative ring and $G$ a group. For an action $\phi\in\text{Hom}(G,\au)$ of the group $G$ on $K$ fixing a subring $k\subseteq K$ elementwise,
we construct two abelian monoids as follows.
\begin{enumerate}
\item The monoid $\CK$ consists of equivariant classes of GCC's of type $\phi$ from $G$ to the Picard groups of $K$-central algebras
(of some bounded cardinality).
\item The monoid $\text{Cliff}_k(\phi)$ consists of equivalence classes of strongly $G$-graded algebras of type $\phi$
up to $G$-graded Clifford system extensions of $K$-central algebras (of bounded cardinality as above).
\end{enumerate}
The torsor role of $H^2(G,K^*_{\phi})$ is described by a natural embedding (see \eqref{Sigma})
$$\Sigma_{\phi}:H^2(G,K^*_{\phi})\hookrightarrow\text{Cliff}_k(\phi).$$ Under the homomorphism $\Sigma_{\phi}$,
the second cohomology group is identified with the $\phi$-type $G$-graded Clifford system extensions of $K$ itself.
More precisely, $\Sigma_{\phi}([\alpha])$ is the class $[K^{\alpha}_{\phi}G]$ of crossed products with respect to any representative $\alpha\in Z^2(G,K^*)$ of $[\alpha]$,
where $G$ acts on $K^*$ via $\phi$.
Two other well-defined homomorphisms of monoids are described. The first homomorphism (see \eqref{chiphi}), $${\chi_{\phi}}:\text{Cliff}_k(\phi)\to\CK$$ sends a
graded class of a strongly graded algebra of type $\phi$ to the equivariance class of its associated GCC.
The second homomorphism of monoids (see \eqref{T-phi}) $${T_{\phi}}:\CK\to H^3(G,K^*_{\phi})$$
sends an equivariance class of a GCC $\Phi$ of type $\phi$ to the obstruction cohomology class $T(\Phi)$.
The homomorphisms $\Sigma_{\phi},\chi_{\phi}$ and $T_{\phi}$ give rise to a realization-obstruction exact sequence of monoids of type $\phi$ as follows.
\begin{Main Theorem}
Let $K$ be a commutative ring and $G$ a group. Then for every $\phi\in\text{Hom}(G,\au)$ there is an exact sequence of abelian monoids of type $\phi$
\begin{equation}\label{introexact}
0\to H^2(G,K^*_{\phi})\stackrel{\Sigma_{\phi}}\to\text{Cliff}_k(\phi)\stackrel{\chi_{\phi}}\to\CK\stackrel{T_{\phi}}\to H^3(G,K^*_{\phi}).
\end{equation}
Furthermore, the graded class of any strongly $G$-graded $k$-algebra appears in a unique sequence \eqref{introexact} (given a cardinality $\kappa$ that bounds its base algebra).

If, additionally, $K$ is an integral domain, then the homomorphism ${T_{\phi}}$ is surjective, terminating the sequence \eqref{introexact} as follows
$$
0\to H^2(G,K^*_{\phi})\stackrel{\Sigma_{\phi}}\to\text{Cliff}_k(\phi)\stackrel{\chi_{\phi}}\to\CK\stackrel{T_{\phi}}\to H^3(G,K^*_{\phi})\to 0.
$$
\end{Main Theorem}

For every $\phi\in\text{Hom}(G,\au)$ there is a sequence \eqref{subexact} of sub-monoids of the exact sequence \eqref{introexact}
describing the strongly graded $k$-algebras of type $\phi$ whose base algebra is $K$ itself.
In particular, the sequence for central graded algebras \cite[Theorem 5.7]{CG03} (with a trivial $\Gamma$-structure) is a special case of \eqref{subexact}.
Another important instance, namely when $K$ is a field and the action $\phi$ is Galois, is in Theorem \ref{commute}.
This theorem presents the restriction-obstruction exact sequence of Brauer groups as an image of a sequence of sub-monoids of \eqref{introexact} (compare with \cite[(45)]{CG03}).

The monoidal structures of $\text{Cliff}_k(\phi)$ is determined by {\it graded products} \eqref{Gprod} between $G$-graded $k$-algebras of type $\phi$.
These products give, in particular, a natural way to twist any $G$-graded $k$-algebra by cocycles (see \eqref{twist}).

The paper is organized as follows. In \S\ref{eqsec} we show how to multiply invertible bimodules over $K$-central algebras of the same $K$-automorphism type $\eta$
such that the product is again an invertible module of type $\eta$.
This multiplication leads to a product of GCC's of type $\phi\in\text{Hom}(G,\au)$, and so to the definition of the monoid $\CK$ as described in \S\ref{actongral}.
In \S\ref{Cse} we construct the graded product of strongly graded algebras of the same type,
and establish the monoid $\text{Cliff}_k(\phi)$, as well as the homomorphisms $\Sigma_{\phi}$ and $\chi_{\phi}$.
In \S\ref{obsection} the obstruction map ${T_{\phi}}$ is shown to be a homomorphism, which is surjective when $K$ is an integral domain.
Our main theorem and some of its particular instances are obtained in \S\ref{ros}. Section \S\ref{Csa} is devoted to Galois actions $\phi$, showing that the restriction-obstruction sequence in
the Brauer Theory is an image of one of these exact sequences.

This paper consists of fairly many notations; we tried to stick to those of \cite{CG01} as much as possible.
There are also fairly many equivalence relations.
The claims that assert independence of the choice of representatives under some of these identifications are left as exercises.

The reader is referred to an extensive list of related literature in \cite{CG03,VO00}.
\section{$K$-central algebras; equivariance of invertible modules}\label{eqsec}
An (associative unital) ring is called a {\it $K$-central algebra} if its center is isomorphic to a given commutative ring $K$.
Let $\mathcal{C}(K)=\mathcal{C}_{\kappa}(K)$ be the set of all $K$-isomorphism classes of $K$-central algebras whose cardinality is bounded by some infinite cardinality $\kappa$.
Denoting by $[R]$ the $K$-isomorphism class of a $K$-central algebra $R$,
it is not hard to verify that $[R_1\otimes_{K}R_2]\in\mathcal{C}(K)$ for every $[R_1],[R_2]\in\mathcal{C}(K)$. Moreover,
$$\begin{array}{ccc}
\mathcal{C}(K)\times\mathcal{C}(K)&\to&\mathcal{C}(K)\\
([R_1],[R_2])&\mapsto&[R_1\otimes_{K}R_2]
\end{array}$$
determines a well-defined operation that furnishes $\mathcal{C}(K)$ with an abelian monoid structure, whose two-sided identity is $[K]$ itself.
\begin{notation}\label{notbimod}
The left and right actions of $R$ on an $R$-bimodule are denoted throughout by $``\star"$ and $``\cdot"$ respectively.
\end{notation}
For any $[R]\in\mathcal{C}(K)$ and any subring $k\subset K$ let $\p(R)$ be the Picard group of isomorphism classes of invertible $R$-bimodules
whose left and right $k$-module structures are the same.
Any $[P]\in\p(R)$ determines a well-defined automorphism $\eta_{P}$ of $K$ via the commuting rule
$$\eta_{P}(x)\star p=p\cdot x,\ \forall x\in K,\ p\in P.$$
Furthermore, there is an exact sequence of groups (see \cite[Prop. II.5.4]{Bass})
\begin{equation}\label{bassequence}
0\to\text{Pic}_K(R)\to\p(R)\xrightarrow{h_R}\au,
\end{equation}
where $\text{Pic}_K(R)$ is naturally embedded in $\p(R)$ and where
\begin{eqnarray}\label{hR}
h_R:\begin{array}{ccc}
\p(R)&\to&\au\\
 ~[P]&\mapsto &\eta_{P}
\end{array}.\end{eqnarray}
In the sequel we make use of the following natural embedding of the outer automorphisms $\text{Out}_k(R):=\text{Aut}_k(R)/\text{Inn}(R)$ inside the Picard group.
For every $\eta\in\text{Aut}_k(R)$ let $R_{1,\eta}$ be an $R$-bimodule, which is free of rank 1 as a left module over $R$, whereas its right $R$-module structure is via $\eta$.
More explicitly, writing the multiplication in $R$ by juxtaposing elements while using the notations $``\star"$ and $``\cdot"$ as above for the left and right actions of $R$ on $R_{1,\eta}$ respectively,
then the $R$-bimodule structure of $R_{1,\eta}$ is given by
\begin{equation}\label{star}r\star s:=r s,\ \ s\cdot r:=s\eta(r),\ \ r\in R,\ \  s\in R_{1,\eta}.\end{equation}
By \cite[Prop. II.5.3]{Bass} there is a well-defined injective morphism
\begin{eqnarray}\label{out}\iota_R:\begin{array}{ccc}\text{Out}_k(R)&\hookrightarrow&\p(R),\\
\eta\cdot\text{Inn}(R)&\mapsto&[R_{1,\eta}]
\end{array}.\end{eqnarray}
We thus consider Out$_k(R)$ as a subgroup of $\p(R)$.

Any $R$-automorphism stabilizes the center, and hence determines an automorphism of $K$ as described using \eqref{hR} and \eqref{out}
\begin{eqnarray}\label{lift}\begin{array}{ccc}
\text{Aut}_k(R) &\to &\text{Aut}_k(K)\\
\eta & \mapsto & h_R([R_{1,\eta}])
\end{array}.\end{eqnarray}
\begin{definition}\label{extendable}
A $K$-automorphism is said to be \textit{extendable} to a $K$-central algebra $R$ if it lies in the image of the morphism \eqref{lift}.
\end{definition}

Let $[P_1]\in\p(R_1)$ and $[P_2]\in\p(R_2)$, where $[R_1],[R_2]\in\mathcal{C}(K)$. Suppose that they are of the same \textit{$K$-automorphism type}, that is
\begin{equation}\label{compat}
\eta:=h_{R_1}[P_1]=h_{R_2}[P_2]\in\au.
\end{equation}
Let
\begin{equation}\label{otimeseta}P_1\otimes_{\eta}P_2:= \bigslant{P_1\otimes_{k}P_2}{W_{\eta}},
\end{equation}
where the $k$-subspace $W_{\eta}\subseteq P_1\otimes_{k}P_2$ is spanned by the elements
$$
\ p_1\cdot x\otimes_kp_2-p_1\otimes_k\eta(x)\star p_2,\ \  \forall p_1\otimes_kp_2\in P_1\otimes_{k}P_2,\ \forall x\in K.$$
The elements in $P_1\otimes_{\eta}P_2$ are denoted by
$\sum_ip^i_1\otimes_{\eta}p^i_2.$

Next, assuming condition \eqref{compat} we define left and right actions of $R_1\otimes_{K}R_2$ on $P_1\otimes_{\eta}P_2$ as follows.
For every $r_1\in R_1,r_2\in R_2,p_1\in P_1,p_2\in P_2$ let
\begin{equation}\label{actwosidel}
(r_1\otimes_Kr_2)\star(p_1\otimes_{\eta}p_2):=r_1\star p_1\otimes_{\eta} r_2\star p_2,
\end{equation}
\begin{equation}\label{actwosider}
(p_1\otimes_{\eta}p_2)\cdot(r_1\otimes_Kr_2):=p_1\cdot r_1\otimes_{\eta} p_2\cdot r_2.
\end{equation}
Then the following claim can easily be verified.
\begin{lemma}\label{struct}
Let $[P_1]\in\p(R_1)$ and $[P_2]\in\p(R_2)$ satisfy \eqref{compat}. Then
Equations \eqref{actwosidel} and \eqref{actwosider} determine an $R_1\otimes_{K}R_2$-bimodule structure on $P_1\otimes_{\eta}P_2$ such that
\begin{enumerate}
\item $[P_1\otimes_{\eta}P_2]\in\p(R_1\otimes_{K}R_2)$.
\item If $[P_1]=[P_1']$ and $[P_2]=[P_2']$ then $[P_1\otimes_{\eta}P_2]=[P'_1\otimes_{\eta}P'_2]$.
\item $h_{R_1\otimes_{K}R_2}[P_1\otimes_{\eta}P_2]=\eta$.
\item If, additionally, $[P_1]\in$Out$_k(R_1)$ and $[P_2]\in$Out$_k(R_2)$ then\\ $[P_1\otimes_{\eta}P_2]\in$Out$_k(R_1\otimes_{K}R_2)$.
\end{enumerate}
\end{lemma}
\begin{definition}\label{defeq}
Let $\psi:R_1\to R_2$ be a $K$-algebra isomorphism of $K$-central algebras $R_1$ and $R_2$ (in particular $[R_1]=[R_2]\in\mathcal{C}(K))$.
Two classes $[P_1]\in\p(R_1),[P_2]\in\p(R_2)$ of invertible bimodules over $R_1$ and $R_2$ respectively
are called $\psi$-{\it equivariant} if there exists
an additive bijection $\varphi:P_1\to P_2$ such that with the notation \ref{notbimod}
$$\varphi(p\cdot r)=\varphi(p)\cdot\psi(r)\ \ \text{and} \ \ \varphi(r\star p)=\psi(r)\star\varphi(p), \ \ \forall r\in R_1,\ \  p\in P_1.$$
The classes $[P_1]\in\p(R_1),[P_2]\in\p(R_2)$ are {\it equivariant} if they are $\psi$-{equivariant} for some $K$-algebra isomorphism $\psi:R_1\to R_2$.
In particular, $[P_1],[P_2]\in\p(R)$ are equivariant if they are $\psi$-{equivariant} for some $\psi\in\Aut_K(R)$.
\end{definition}
\begin{notation}
We denote the equivariance class of $[P]\in\p(R)$ by $[P]_{\e}$.
Running over all $[R]\in\mathcal{C}(K)$ we write
\begin{equation}\label{P}\mathcal{P}_k(K):=\{[P]_{\e}|\ \ [P]\in\p(R),\ \  [R]\in\mathcal{C}(K)\}.
\end{equation}
This is the set of equivariance classes of all invertible modules over central $K$-algebras (of bounded cardinality as above).
The equivariance classes of modules over $K$ itself are denoted by
\begin{equation}\label{P1}\mathcal{P}_1(K):=\left\{[P]_{\e}\ \ |\ \ {[P]\in\p (K)}\right\}\subseteq\mathcal{P}_k(K).\end{equation}
\end{notation}
It is important to emphasize that all the members of a given equivariance class $[P]_{\e}$ belong to the Picard groups of isomorphic $K$-central algebras. In other words,
there is a well-defined map
\begin{equation}\label{eq21}
\mathcal{P}_k(K)\to\mathcal{C}(K),
\end{equation}
taking the equivariance class $[P]_{\e}$ of $[P]\in\p(R)$ to the isomorphism class $[R]$.

We also remark that the equivariance class of an outer automorphism, or rather its embedding \eqref{out} in the Picard group, consists solely of such outer automorphism embeddings.
That is, if $[P_1]\in\text{Out}_k(R_1)$ and $[P_2]\in\p(R_2)$, then
$$[P_1]_{\e}=[P_2]_{\e} \Rightarrow [P_2]\in\text{Out}_k(R_2).$$
The following straightforward claim says that equivariant classes agree on the homomorphisms \eqref{hR}.
\begin{lemma}\label{agree}
Suppose that the classes $[P_1]\in\p(R_1),[P_2]\in\p(R_2)$ are equivariant. Then $h_{R_1}[P_1]=h_{R_2}[P_2]$.
\end{lemma}
By Lemma \ref{agree}, the homomorphisms \eqref{hR} give rise to a well-defined map
\begin{eqnarray}\label{h}
h:\begin{array}{ccc}\mathcal{P}_k(K)&\to &\au\\
~[P]_{\e}&\mapsto & h_R([P])\end{array},\ \ [P]\in\p(R),\ \  [R]\in\mathcal{C}(K).\end{eqnarray}
Let $\eta\in\au$, it is easy to check that the class $[K_{1,\eta}]\in \p (K)$ (see \eqref{out}) satisfies $h([K_{1,\eta}]_{\e})=\eta.$
The following claim is deduced from Lemma \ref{struct}.
\begin{lemma}\label{h-1}
Let $\eta\in \au$. Then there is a well-defined operation
\begin{eqnarray}\label{opeq}\begin{array}{ccc}
h^{-1}(\eta)\times h^{-1}(\eta)&\to&h^{-1}(\eta)\\
~[P_1]_{\e}\otimes_{\eta}[P_2]_{\e}&:=&[P_1\otimes_{\eta}P_2]_{\e}
\end{array},\end{eqnarray}
which turns $h^{-1}(\eta)(\subseteq \mathcal{P}_k(K))$ into an abelian monoid, whose identity is $[K_{1,\eta}]_{\e}$.
\end{lemma}
Consequently, the map \eqref{h} partitions $\mathcal{P}_k(K)=\bigsqcup_{\eta\in\au}h^{-1}(\eta)$ as a disjoint union of abelian monoids of type $\eta$.

Note that if $[P_1],[P_2]\in\p (K)\cap  h^{-1}_K(\eta)$ then $$[P_1\otimes_{\eta}P_2]\in\p (K\otimes_KK)\cap h_{K\otimes_KK}^{-1}(\eta).$$
With the notation \eqref{P1} we obtain that
\begin{equation}\label{sub1}
\mathcal{P}_1(K)\cap h^{-1}(\eta)
\end{equation}
is a sub-monoid of $h^{-1}(\eta).$
\section{Generalized collective characters}\label{actongral}
Recall from the introduction that, with the terminology of \cite{CG01}, a group homomorphism from a group $G$ to $\p(R)$ is called a {\it Generalized Collective Character} (GCC).
\begin{definition}\label{equiGCC}
Let $[R_1]=[R_2]\in\mathcal{C}(K)$.
We say that two GCC's $\Phi_1\in$Hom$(G,\p(R_1))$ and $\Phi_2\in$Hom$(G,\p(R_2))$ are {\it equivariant} if there exists a $K$-algebra isomorphism $\psi:R_1\to R_2$ such that
$\Phi_1(g)$ and $\Phi_2(g)$ are $\psi$-equivariant (see Definition \ref{defeq}) for every $g\in G$.
\end{definition}
Denote the equivariance class of a GCC $\Phi\in$Hom$(G,\p(R))$ by $[\Phi]_{\e}$, and let
\begin{equation}\label{CGK}
\CGK:=\left\{[\Phi]_{\e}|\ \ \Phi\in\text{Hom}(G,\p(R)),\ \ [R]\in\mathcal{C}(K)\right\}.
\end{equation}

Notice again that, similarly to \eqref{eq21}, there is a well-defined map
\begin{equation}\label{eq22}
\CGK\to\mathcal{C}(K),
\end{equation}
taking the equivariance class $[\Phi]_{\e}$ of $\Phi\in$Hom$(G,\p(R))$ to the isomorphism class $[R]$ of the corresponding $K$-central algebra $R$.

Let us move to the well-defined functorial map associated to \eqref{h}
\begin{eqnarray}\label{h*}\begin{array}{rcl}
h^*:\CGK&\to& \text{Hom}(G,\au)\\
h^*([\Phi]_{\e}):g&\mapsto&h([\Phi(g)]_{\e})
\end{array},\ \ [\Phi]_{\e}\in\CGK, g\in G.\end{eqnarray}
An equivariance class $[\Phi]_{\e}\in\CGK$ is \textit{of type} $h^*([\Phi]_{\e})$.
For a group action $\phi\in\text{Hom}(G,\au)$ of $G$ on $K$ which fixes $k$, let %$$\CK:=(h^*)^{-1}(\phi)\subseteq\bigsqcup_{[R]\in\mathcal{C}(K)}\text{Hom}(G,\p(R)),$$ that is
\begin{equation}\label{CK}
\CK:=(h^*)^{-1}(\phi)=\left\{[\Phi]_{\e}\in\CGK |\ \ h([\Phi(g)]_{\e})=\phi(g), \forall g\in G\right\}
\end{equation}
be the set of all equivariance classes of type $\phi$.
We endow $\CK$ with an abelian monoid structure as follows.
Let $\Phi_1\in\text{Hom}(G,\p(R_1))$ and $\Phi_2\in\text{Hom}(G,\p(R_2))$ be two GCC's representing the equivariance classes $[\Phi_1]_{\e},[\Phi_2]_{\e}\in\CK$ respectively.
For every $g\in G$, let $M_1(g)$ and $M_2(g)$ be invertible bimodules over $R_1$ and $R_2$ respectively such that $[M_1(g)]=\Phi_1(g)$ and $[M_2(g)]=\Phi_2(g)$.
Define
\begin{eqnarray}\label{prodck}\Phi_1\otimes_{\phi}\Phi_2:\begin{array}{ccl}
G&\to&\p (R_1\otimes_KR_2)\\
g&\mapsto &[M_1(g)\otimes_{\phi(g)}M_2(g)]
\end{array},\ \  g\in G.
\end{eqnarray}
Then it is not hard to check the following
\begin{lemma}\label{otimesphi} Let $\phi\in\text{Hom}(G,\au)$ be a $G$-action on $K$. Then with the notation \eqref{prodck}
\begin{enumerate}
\item $\Phi_1\otimes_{\phi}\Phi_2$ does not depend on the choice of representatives $M_1$ and $M_2$.
\item $\Phi_1\otimes_{\phi}\Phi_2$ is a group homomorphism.
\item $[\Phi_1\otimes_{\phi}\Phi_2]_{\e}\in\CK$.
\item If $[\Phi_1]_{\e}=[\Phi_1']_{\e}$ and $[\Phi_2]_{\e}=[\Phi'_2]_{\e}$, then $[\Phi_1\otimes_{\phi}\Phi_2]_{\e}=[\Phi'_1\otimes_{\phi}\Phi'_2]_{\e}$.
\item If, additionally, $\Phi_1$ and $\Phi_2$ are collective characters then so is $\Phi_1\otimes_{\phi}\Phi_2$.
\end{enumerate}
\end{lemma}
By Lemma \ref{otimesphi} there is a well-defined product in $\CK$ given by
\begin{eqnarray}\label{prodeq}\begin{array}{ccc}
\CK\times\CK&\to &\CK\\
~[\Phi_1]_{\e}\circ[\Phi_2]_{\e}&:=&[\Phi_1\otimes_{\phi}\Phi_2]_{\e}
\end{array}.\end{eqnarray}
Next, note that
\begin{eqnarray}\label{1phi}1_{\phi}:\begin{array}{c}
G\to \p(K)\\
g\mapsto [K_{1,\phi(g)}]
\end{array}\end{eqnarray} is a group homomorphism such that $[1_{\phi}]_{\e}\in \CK$. We have
\begin{lemma}
Let  $\phi\in\text{Hom}(G,\au)$. Then the operation \eqref{prodeq} turns $\CK$ to an abelian monoid whose
identity is the equivariance class $[1_{\phi}]_{\e}$. Moreover, the restriction of \eqref{eq22} to $\CK$, namely
\begin{equation}\label{Psiphi}
\Psi_{\phi}:\begin{array}{ccc}\CK&\to&\mathcal{C}(K)\\
~[\Phi]_{\e}&\mapsto &[R]
\end{array}, \Phi\in\text{Hom}(G,\p(R))\end{equation}
is a well-defined morphism of monoids.
\end{lemma}
\begin{proof}
By direct computation.
\end{proof}
%\begin{remark}%
Consequently, the map \eqref{h*} partitions
\begin{equation}\label{partition1}
\CGK=\bigsqcup_{\phi\in\text{Hom}(G,\au)}\CK
\end{equation}
as a disjoint union of monoids with respect to the type of their members.%\end{remark}
\begin{lemma}
Let  $\phi\in\text{Hom}(G,\au)$. Then with the above notation the following are sub-monoids of $\CK$
\begin{equation}\label{sub2}\CPK:=\Psi_{\phi}^{-1}([K])=\CK\cap\left\{[\Phi]_{\e}|\ \ \Phi\in\text{Hom}(G,\p (K))\right\}<\CK,\end{equation}
and
\begin{equation}\label{cpsub}\CO:=\left\{[\Phi]_{\e}\in\CK|\ \ \Phi\in\text{Hom}(G,\text{Out}_k (R)),\ \ [R]\in\mathcal{C}(K)\right\}<\CK.\end{equation}
\end{lemma}
\begin{proof}
These are consequences of equation \eqref{sub1} and from Lemma \ref{otimesphi}(5) respectively.
\end{proof}
The image of the sub-monoid $\CO$ under the morphism \eqref{Psiphi} lies in the sub-monoid of $\mathcal{C}(K)$ of classes of normal $K$-central algebras as follows.
\begin{definition}\label{normal} (compare with \cite{EM})
A $K$-central algebra $R$ is \textit{normal}, or \textit{invariant}, with respect to an action $\phi\in\text{Hom}(G,\au)$
if for every $g\in G$ the $K$-automorphism $\phi(g)$ is extendable to $R$ (see Definition \ref{extendable}).
\end{definition}
The set of isomorphism classes of $\phi$-normal $K$-central algebras is a sub-monoid denoted by $\mathcal{C}^{\phi}(K)<\mathcal{C}(K)$. We have
\begin{lemma}\label{resnormal}
With the notation \eqref{Psiphi}
\begin{equation}\label{PsiphiCO}
\Psi_{\phi}(\CO)\subseteq \mathcal{C}^{\phi}(K).
\end{equation}
\end{lemma}
\begin{proof}
Let $\Phi\in\text{Hom}(G,\text{Out}_k (R))$ be any collective character such that $[\Phi]_{\e}\in \CO$.
In particular, $\Phi$ is of type $\phi$, that is $h_R(\Phi)(g)=\phi(g)$ for every $g\in G$.
Then any automorphism $\eta\in$Aut$_k(R)$ with ~$\eta\cdot$Inn$(R)=\Phi(g)\in$Out$_k(R)$
is an extension of $\phi(g)$ to an $R$-automorphism (see \eqref{lift}). Hence $R$ is $\phi$-normal and so $[R]=\Psi_{\phi}([\Phi]_{\e})\in\mathcal{C}^{\phi}(K)$.
\end{proof}

\section{Clifford system extensions}\label{Cse}
A $k$-algebra $A$ is {\it strongly $G$-graded} if it admits an additive decomposition
\begin{equation}\label{strongly}
A=\oplus_{g\in G}A_g
\end{equation}
such that $A_gA_h=A_{gh}$ for every $g,h\in G$. This multiplicative condition yields that the trivial homogeneous component $A_e$ is in particular a subalgebra of $A$ called the {\it base algebra}.
Moreover, every homogeneous component $A_g$ is an invertible $A_e$-bimodule and \eqref{strongly} describes a decomposition of $A$ as a direct sum of $A_e$-bimodules.
A strongly $G$-graded $k$-algebra \eqref{strongly} is a {\it crossed product} if there exists a unit $u_g\in (A_g)^*$ for every $g\in G$.
A crossed product is written as
$$A_e*G:=\bigoplus_{g\in G}A_e\cdot u_g=\bigoplus_{g\in G}u_g\cdot A_e.$$
When the base algebra of a crossed product $k$-algebra is exactly $K$ then this crossed product $\bigoplus_{g\in G}\text{Span}_K\{u_g\}$ is determined by a $G$-action $\phi\in\text{Hom}(G,\au)$ on $K$
given by the conjugation $\phi(g):x\mapsto u_gxu_g^{-1}$
and a two-place function
$$\alpha:\begin{array}{ccc}G\times G & \to & K^*\\
(g,h)& \mapsto & u_gu_hu_{gh}^{-1}\end{array}.$$
The associativity of this crossed product, which we denote by $K_{\phi}^{\alpha}G$, implies that $\alpha$ is a 2-cocycle,
that is $\alpha\in Z^2(G,K_{\phi}^*)$, where the $G$-module structure of $K^*$ is via the action $\phi$.
When $\alpha=1\in Z^2(G,K_{\phi}^*)$ is the trivial cocycle, the corresponding crossed product is called a {\it skew group algebra}.
\begin{examples}
(a) The group algebra $KG$ is an immediate example for a crossed product.
In particular, the group algebra with respect to the group of integers $G=\Z$ can be identified with the Laurent polynomial algebra
$K[X^{\pm 1}]$ (see \cite[Example 1.3.2]{NVO04}). Adding an action $\phi\in\text{Hom}(\Z,\au)$
yields the corresponding {\it skew Laurent polynomial algebra} $K[X^{\pm 1};\phi]$ (see \cite[\S 1.2]{MR}).
Note that the $\Z$-graded algebras $K[X^{}]$ and, more generally, $K[X^{};\phi]$ are not strongly graded.\\
(b) Perhaps the simplest way to demonstrate a strongly graded algebra which is not a crossed product is by a $\Z/2=\{\bar{0},\bar{1}\}$-grading of the matrix algebra
$$\text{M}_3(K)=(\text{M}_3(K))_{\bar{0}}\oplus(\text{M}_3(K))_{\bar{1}}$$
where
(see \cite[Example 1.3.5]{NVO04})
$$
\begin{array}{cc}
(\text{M}_3(K))_{\bar{0}}:=
\left(
\begin{array}{ccccc}
K & K & 0  \\
K & K &  0 \\
0 & 0 & K
\end{array} \right)\text{ and }
& (\text{M}_3(K))_{\bar{1}}:=
\left(
\begin{array}{ccccc}
0  & 0 & K \\
0 &  0 & K \\
K & K& 0
\end{array} \right).
\end{array}$$
It is not hard to verify that $(\text{M}_3(K))_{\bar{0}}=(\text{M}_3(K))_{\bar{1}}\cdot(\text{M}_3(K))_{\bar{1}}$, and that $(\text{M}_3(K))_{\bar{1}}$ admits no invertible elements.
Consequently, $\text{M}_3(K)$ is strongly $\Z/2$-graded but not a crossed product.
\end{examples}

There are few equivalence relations on graded algebras. Let us mention two of them.
A {\it graded isomorphism} (see, e.g., \cite[Definition 2.3]{GS16})
between \eqref{strongly} and another strongly $G$-graded $k$-algebra $A'=\oplus_{g\in G}A'_g$ is
a $k$-algebra isomorphism $\psi:A\to A'$ such that $\psi(A_g)=A'_g$ for every $g\in G$.
A graded isomorphism takes homogeneous units to homogeneous units of the same degree, hence a graded class of a crossed product consists solely of crossed products.
The second equivalence relation, whose origins are in \cite{D70}, refines the graded isomorphism relation as follows.
\begin{definition}\label{cliffeq}\cite[Definition 2.1]{CG01}
Let $R$ be a $k$-algebra and let $G$ be a group.
A {\it $G$-graded $k$-Clifford system extension of $R$} is a pair $(A,j)$ where $A$ is a strongly $G$-graded $k$-algebra \eqref{strongly} and
$j:R\hookrightarrow A$ is a $k$-algebra embedding with $j(R)=A_e$. Two $G$-graded Clifford system extensions $(A,j)$ and $(A',j')$ of $R$ are equivalent if there exists
a graded isomorphism $\psi:A\to A'$ such that $\psi\circ j=j'$.
\end{definition}
Let Cliff$_k(G,R)$ denote the set of equivalence classes of $G$-graded Clifford system extensions of a $k$-algebra $R$ (see \cite[(1)]{CG01}).
For short, we denote the equivalence class of a pair $(A,j)$ just by $[A]$.

Every strongly graded algebra \eqref{strongly} admits an associated GCC \cite[(2)]{CG01} which assigns the invertible class $[A_g]\in$Pic$_k(A_e)$ to a group element $g\in G$.
Clearly, graded isomorphic algebras, and hence also equivalent Clifford system extensions, admit the same associated GCC.
That is, there is a well-defined map
\begin{eqnarray}\label{GCC}\chi_R:\begin{array}{ccc}\text{Cliff}_k(G,R)&\to&\text{Hom}(G,\p(R))\\
\chi_R[A]:g&\mapsto &[A_g]\end{array}.\end{eqnarray}
Crossed products are characterized by the map \eqref{GCC}.
\begin{lemma}\label{CPOut} (see \cite[\S 3]{CG01})
A strongly graded algebra \eqref{strongly} is a crossed product if and only if its associated GCC is a collective character.
\end{lemma}
For a commutative $k$-algebra $K$, let $\text{Cliff}_k(G,\mathcal{C}(K))$ denote the set of equivalence classes of $G$-graded $k$-Clifford system extensions of the $K$-central algebras
(of cardinality bounded by $\kappa$).
Then the map \eqref{GCC} yields, in turn, a well-defined map to the set \eqref{CGK} of GCC equivariance classes
\begin{eqnarray}\label{chi}\chi:\begin{array}{rcl}
\text{Cliff}_k(G,\mathcal{C}(K))&\to&\CGK\\
\chi([\bigoplus _{g\in G} A_g]):g&\mapsto&[A_g]_{\text{eq}}.
\end{array}\end{eqnarray}
Crossed products over the base algebra $K$ are characterized using \eqref{1phi} and \eqref{chi} as follows.
\begin{lemma}\label{char}
A strongly $G$-graded $k$-algebra \eqref{strongly} is graded isomorphic to $K_{\phi}^{\alpha}G$ for some $\alpha\in Z^2(G,K^*)$ iff $[A]\in$Cliff$_k(G,K)$ and $\chi[A]=[1_{\phi}]_{\e}$.
\end{lemma}
Fix $\phi\in\text{Hom}(G,\au)$. With the notations \eqref{CK} and \eqref{chi} let
$$\text{Cliff}_k(\phi):=\chi^{-1}(\CK)\subseteq\text{Cliff}_k(G,\mathcal{C}(K)).$$
A Clifford system extension class in $\text{Cliff}_k(\phi)$ is said to be \textit{of type $\phi$}.
A product in $\text{Cliff}_k(\phi)$ can be defined as follows
Two strongly $G$-graded $k$-algebras
$$A=\bigoplus _{g\in G} A_g,\ \ A'=\bigoplus _{g\in G} A'_g$$
with $[A],[A']\in \text{Cliff}_k(\phi)$ determine a new $G$-graded algebra
\begin{equation}\label{Gprod}
A\otimes_{\phi}A'=\bigoplus_{g\in G} (A\otimes_{\phi}A')_g,
\end{equation}
where the homogeneous components of the product are defined using \eqref{otimeseta}
$$(A\otimes_{\phi}A')_g:=A_g\otimes_{\phi(g)}A'_g,$$
and where the multiplication in $A\otimes_{\phi}A'$ is given by
$$(a_g\otimes_{\phi(g)} a'_g)(a_h\otimes_{\phi(h)} a'_h):=a_ga_h\otimes_{\phi(gh)} a'_ga'_h\in(A\otimes_{\phi}A')_{gh}$$ for
$a_g\otimes_{\phi(g)} a'_g\in (A\otimes_{\phi}A')_g$ and $a_h\otimes_{\phi(h)} a'_h\in (A\otimes_{\phi}A')_h.$
It is not hard to verify that
\begin{lemma} With the above notation $A\otimes_{\phi}A'$ is a $G$-graded $k$-Clifford system extension of the $K$-central algebra $A_e\otimes_{K}A'_e$ with
$[A\otimes_{\phi}A']\in\text{Cliff}_k(\phi)$. Moreover,
\begin{eqnarray}\label{prodcliff}\begin{array}{ccc}
\text{Cliff}_k(\phi)\times\text{Cliff}_k(\phi)&\to&\text{Cliff}_k(\phi)\\
([A],[{A'}])&\mapsto&[A\otimes_{\phi}{A'}]
\end{array}\end{eqnarray}
determines a well-defined abelian monoid structure on $\text{Cliff}_k(\phi)$, whose identity is the grading class of the corresponding skew group algebra $K^1_{\phi} G$.
\end{lemma}

Similarly to \eqref{partition1}, the map \eqref{chi} partitions
\begin{equation}\label{partition2}
\text{Cliff}_k(G,\mathcal{C}(K))=\bigsqcup_{\phi\in\text{Hom}(G,\au)}\text{Cliff}_k(\phi)
\end{equation}
as a disjoint union of monoids with respect to the type of their members.
\begin{lemma}\label{chiphilemma}
The restriction
\begin{equation}\label{chiphi}
\chi_{\phi}:\text{Cliff}_k(\phi)\to\CK.
\end{equation}
of the map $\chi$ \eqref{chi} to $\text{Cliff}_k(\phi)$ is a morphism of abelian monoids for every $\phi\in\text{Hom}(G,\au)$.
\end{lemma}
\begin{proof}
By direct computation.
\end{proof}
By Lemma \ref{CPOut}, the classes of crossed products in $\text{Cliff}_k(\phi)$ are exactly those classes which are mapped under \eqref{chi} to the collective characters equivariance classes
$\CO$ (see \eqref{cpsub}). They form a sub-monoid, which we denote by
\begin{equation}\label{EX}\EX:=\chi^{-1}(\CO)<\text{Cliff}_k(\phi).\end{equation}
Note that with the notations \eqref{EX} and \eqref{chiphi} we have
\begin{equation}\label{EXCO}
\chi_{\phi}(\EX)\subseteq\CO.
\end{equation}
Let $R*G=\bigoplus_{g\in G}R\cdot u_g$ be crossed product with $[R*G]\in\EX$. Then conjugations by the invertible homogeneous elements $\{u_g\}_{g\in G}$ yield an extension of the $G$-action
$\phi$ to $R$. We record this observation for the sequel  as follows.
\begin{lemma}\label{Re}
Let $\phi\in\text{Hom}(G,\au)$. Then $\Psi_{\phi}\circ\chi_{\phi}([R*G])=[R]$ for every $[R*G]\in\EX$. In particular, if $R*G$ is a crossed product of type $\phi$, then
the $K$-central algebra $R$ is $\phi$-normal.
\end{lemma}
\begin{proof}
This is a consequence of Lemma \ref{resnormal} and equation \eqref{EXCO}.
\end{proof}
It is not hard to verify that the family of equivalence classes
in $\text{Cliff}_k(\phi)$ of $G$-graded Clifford extensions of $K$ itself is a sub-monoid
\begin{equation}\label{submon}
\text{Cliff}_k(\phi)\cap\text{Cliff}_k(G,K)<\text{Cliff}_k(\phi).
\end{equation}
With the notation \eqref{sub2} we have
\begin{equation}\label{sub3}
\chi_{\phi}(\text{Cliff}_k(\phi)\cap\text{Cliff}_k(G,K))\subseteq\CPK.
\end{equation}
Notice that the invertible elements in $\text{Cliff}_k(\phi)$ lie in this sub-monoid, that is
$$\text{Cliff}_k(\phi)^*\subseteq\text{Cliff}_k(\phi)\cap\text{Cliff}_k(G,K).$$
\begin{lemma}\label{cliff}
Let $\phi\in\text{Hom}(G,\au)$. Then \begin{enumerate}
\item $[K_{\phi}^{\alpha}G\otimes_{\phi}K_{\phi}^{\beta}G]=[K_{\phi}^{\alpha\beta}G]$ for every $\alpha,\beta\in Z^2(G,K_{\phi}^*)$.
In particular, $[K_{\phi}^{\alpha}G]\in \text{Cliff}_k(\phi)^*$ for every $\alpha\in Z^2(G,K_{\phi}^*)$.
\item $[K_{\phi}^{\alpha}G]=[K_{\phi}^{\beta}G]\in \text{Cliff}_k(G,K)$ iff $[\alpha]=[\beta]\in H^2(G,K_{\phi}^*)$.
\end{enumerate}
\end{lemma}
\begin{proof}
Let
$$K_{\phi}^{\alpha}G=\bigoplus_{g\in G}\text{Span}_K\{u_g\},\ \ K_{\phi}^{\beta}G=\bigoplus_{g\in G}\text{Span}_K\{v_g\}\text{, and }K_{\phi}^{\alpha\beta}G=\bigoplus_{g\in G}\text{Span}_K\{w_g\}.$$
It is not hard to check that the $K$-linear map determined by
$u_g\otimes_{\phi(g)}v_g\mapsto w_g$ is an equivalence of $G$-graded Clifford $K$-extensions proving the first claim of the lemma.

As for the second claim,
note that two classes $[A_1],[A_2]\in\text{Cliff}_k(G,K)$ are equal iff there exists a graded isomorphism $\psi:A_1\to A_2$ such that $\psi|_K=$Id$_K$.
A straightforward computation shows that the following two conditions are equivalent: \\
(a) There exists a subset $\{\lambda_g\}_{g\in G}\subset K^*$ such that $K$-linear map $\psi:K_{\phi}^{\alpha}G\to K_{\phi}^{\beta}G$ determined by $u_g\mapsto\lambda_gv_g$ is multiplicative.\\
(b) There exists a subset $\{\lambda_g\}_{g\in G}\subset K^*$ such that
$$\alpha(g,h)=\lambda_g\phi(g)(\lambda_h)\lambda^{-1}_{gh}\beta(g,h),\ \ \forall g,h\in G.$$
Condition (a) is the existence of an equivalence $\psi$ of $G$-graded Clifford system $K$-extensions between $K_{\phi}^{\alpha}G$ and $K_{\phi}^{\beta}G$, whereas (b) is the cohomology equivalence of
the 2-cocycles $\alpha,\beta\in Z^2(G,K_{\phi}^*)$.
\end{proof}
Lemma \ref{cliff} yields
\begin{lemma}\label{units}
Let $\phi\in\text{Hom}(G,\au)$. Then
\begin{eqnarray}\label{Sigma}\Sigma_{\phi}:\begin{array}{ccl}H^2(G,K^*_{\phi})&\to &\text{Cliff}_k(\phi)^*\\
~[\alpha]&\mapsto &[K_{\phi}^{\alpha}G]\end{array}
\end{eqnarray} determines a well-defined injective morphism of abelian groups.
\end{lemma}
\begin{remark}
Compare \eqref{Sigma} with the \textit{Crossed Product Construction} \cite{AR89} which embeds $H^2(G,K^*_{\phi})$ in the monoid of all isomorphism classes of $K/k$ algebras
with the Sweedler multiplication \cite{S74}.
\end{remark}
Let $\phi\in\text{Hom}(G,\au)$, by Lemma \ref{char} we have
\begin{lemma}\label{kerim}
With the notations of \eqref{chiphi} and \eqref{Sigma}
\begin{equation}\label{kerchi}
Ker(\chi_{\phi})=\text{Im}(\Sigma_{\phi}).
\end{equation}\end{lemma}
\subsection{Realizability}\label{realizable}
The map \eqref{GCC} yields the following notion of realizable GCC's.
\begin{definition}(see \cite[\S 2]{CG01})
A GCC $\Phi\in\text{Hom}(G,\p(R))$ is {\it realizable} if it lies in the image of $\chi_R$.
\end{definition}
The definitions
of the maps \eqref{chi} and \eqref{chiphi} immediately yield equivalent conditions for realizability, recorded here for later use.
\begin{lemma}\label{equiv1}
Given $[\Phi]_{\e}\in\CK$,
the following are equivalent.
\begin{enumerate}
\item $[\Phi]_{\e}\in$Im$(\chi_{\phi})$.
\item $[\Phi]_{\e}\in$Im$(\chi)$.
\item There exist $[R]\in\mathcal{C}(K)$ and a GCC $\Phi\in$Hom$(G,\p(R))$ in the equivariance class $[\Phi]_{\e}$ such that $\Phi$ is realizable.
\end{enumerate}
\end{lemma}
\subsection{Twisted graded algebras}\label{Tga}
For any $\alpha\in Z^2(G,K^*_{\phi})$ and a strongly $G$-graded $k$-algebra \eqref{strongly}
%with $[A]\in\text{Cliff}_k(\phi)$
of type $\phi$ write
\begin{equation}\label{twist}
\alpha(A):=K^{\alpha}_{\phi} G\otimes_{\phi}A.
\end{equation}
The $G$-graded algebra $\alpha(A)$ can be regarded as a {\it twisted graded algebra}.
%(compare with symmetric twists for graded algebras over abelian groups \cite[Definition 3.1]{E18}).

We can think of $H^2(G,K^*_{\phi})$ as acting on $\text{Cliff}_k(\phi)$ via the twisting
\begin{equation}\label{Z2actgr}
[\alpha]([A]):=[\alpha(A)]=[K^{\alpha}_{\phi}G\otimes_{\phi}A]
\end{equation}
for every $[\alpha]\in H^2(G,K^*_{\phi})$ and $[A]\in\text{Cliff}_k(\phi)$. The base algebra $(K^{\alpha}_{\phi}G\otimes_{\phi}A)_e$ is isomorphic to $A_e$.
In particular, for every $[R]\in\mathcal{C}(K)$,
the cohomology group $H^2(G,K^*_{\phi})$ acts by twistings on the set $\chi_R^{-1}(\phi)=\text{Cliff}_k(\phi)\cap\text{Cliff}_k(G,R)$.

\section{The obstruction map}\label{obsection}
Let $[R]\in\mathcal{C}(K)$ and $\phi\in\text{Hom}(G,\au)$. Every GCC $\Phi\in$Hom$(G,\p(R))$ of type $\phi$ (that is $h_R(\Phi)=\phi$)
determines an obstruction cohomology class $T(\Phi)\in H^3(G,K^*_{\phi})$. Here is the description of \cite[\S 3]{CG01}, adapted to our context.\\
{\it In each isomorphism class $\Phi(\sigma)\in\p(R)$ ($\sigma\in G$), choose an invertible $R$-bimodule $P_{\sigma}$
selecting $P_e=R.$ Since $\Phi$ is a homomorphism, the bimodules $P_{\sigma}\otimes_RP_{\tau}$ and $P_{\sigma\tau}$ must be isomorphic for each pair $\sigma,\tau\in G$.
Then we can select bimodule isomorphisms
$$\Gamma_{\sigma,\tau}:P_{\sigma}\otimes_RP_{\tau}\to P_{\sigma\tau}$$ with
\begin{equation}\label{twoeqs}
\Gamma_{\sigma,e}(p_{\sigma}\otimes_R r)=p_{\sigma}\cdot r \ \ \text{  and  }\ \
\Gamma_{e,\sigma}(r\otimes_R p_{\sigma})=r\star p_{\sigma},\ \ r\in R=P_e,\ \ p_{\sigma}\in P_{\sigma}.
\end{equation}
By \cite[Lemma 3.1]{CG01}, for every $\sigma,\tau,\gamma\in G$ there exists a unique element $T_{\sigma,\tau,\gamma}^{\Phi}\in K^*$ such that for every
$p_{\sigma}\in P_{\sigma}, p_{\tau}\in P_{\tau}$ and $ p_{\gamma}\in P_{\gamma}$
\begin{equation}\label{TPhi}
\Gamma_{\sigma\tau,\gamma}(\Gamma_{\sigma,\tau}(p_{\sigma}\otimes_Rp_{\tau})\otimes_R p_{\gamma})=
T^{\Phi}_{\sigma,\tau,\gamma}\star\Gamma_{\sigma,\tau\gamma}(p_{\sigma}\otimes_R\Gamma_{\tau,\gamma}(p_{\tau}\otimes_Rp_{\gamma}))\in P_{\sigma\tau\gamma}.
\end{equation}
\begin{theorem}\label{CGthm}(Cegarra and Garz\'on \cite{CG01})\begin{enumerate}
\item The cochain $T^{\Phi}:G^3\to K^*$ is a 3-cocycle with coefficients in the $G$-module $K_{\phi}^*$, where $\phi:=h_R(\Phi)\in\text{Hom}(G,\au)$.
\item Other choices of bimodule isomorphisms $\{\Gamma_{\sigma,\tau}\}_{\sigma,\tau\in G}$ and of $R$-bimodules $\{P_{\sigma}\}_{\sigma\in G}$
leave the cohomology class $T(\Phi):=[T^{\Phi}]\in H^3(G,K^*_{\phi})$ unaltered.
\item A GCC $\Phi\in$Hom$(G,\p(R))$ is realizable (\S\ref{realizable}) if and only if $T(\Phi)=0\in H^3(G,K^*_{\phi})$.
\end{enumerate}\end{theorem}
}
Back to our definition of the monoid $\CK$ we have
\begin{lemma}\label{TPwdmorphism}
The correspondence
\begin{eqnarray}\label{T-phi}T_{\phi}:\begin{array}{rcc}
\CK&\to &H^3(G,K^*_{\phi})\\
~[\Phi]_{\e}&\mapsto& T(\Phi)
\end{array}\end{eqnarray}
determines a well-defined homomorphism of monoids.
\end{lemma}
\begin{proof}
The proof has two parts: (a) $T_{\phi}$ is well-defined. (b) $T_{\phi}$ is a morphism of monoids.
(a) We first show that $T_{\phi}$ is independent of the choice of equivariant GCC's.
Let $\Phi\in$Hom$(G,\p(R))$ be a GCC with choices of $\{\Gamma_{\sigma,\tau}\}_{\sigma,\tau\in G}$ and $\{P_{\sigma}\}_{\sigma\in G}$ as above.
Let $\Phi'\in$Hom$(G,\p(R'))$ be a GCC which is equivariant to $\Phi\in$Hom$(G,\p(R))$, with corresponding $\{P'_{\sigma}\}_{\sigma\in G}$
(the isomorphisms $\{\Gamma'_{\sigma,\tau}\}_{\sigma,\tau\in G}$ will be defined in \eqref{Gamma'}).
By Definition \ref{equiGCC} there exists a $K$-algebra isomorphism $\psi:R\to R'$ and bijections $\varphi_{\sigma}:P_{\sigma}\to P'_{\sigma}$ for every $\sigma\in G$
with $\varphi_e=\psi$, such that for every $r\in R,   p_{\sigma}\in P_{\sigma}$, and $\sigma\in G$
\begin{equation}\label{biject}\varphi_{\sigma}(p_{\sigma}\cdot r)=\varphi_{\sigma}(p_{\sigma})\cdot\psi(r)\ \ \  \text{and} \ \ \  \varphi_{\sigma}(r\star p_{\sigma})
=\psi(r)\star\varphi_{\sigma}(p_{\sigma})   .\end{equation}
Then it is not hard to check that
\begin{eqnarray}\label{Gamma'}\Gamma'_{\sigma,\tau}:\begin{array}{ccl}
P'_{\sigma}\otimes_{R'} P'_{\tau}&\to &P'_{\sigma\tau}\\
p'_{\sigma}\otimes_{R'} p'_{\tau}&\mapsto&\varphi_{\sigma\tau}(\Gamma_{\sigma,\tau}(\varphi^{-1}_{\sigma}(p'_{\sigma})\otimes_R\varphi^{-1}_{\tau}(p'_{\tau})))
\end{array}
\end{eqnarray}
are $R'$-bimodule isomorphisms for every $\sigma,\tau\in G$, satisfying the demands \eqref{twoeqs} on $\Gamma'_{\sigma,e}$ and $\Gamma'_{e,\sigma}$.
We now show that with these choices, the 3-cocycles which correspond to $\Phi$ and $\Phi'$ are the same.
As in \eqref{TPhi}, $T_{\sigma,\tau,\gamma}^{\Phi'}\in K^*$ is defined by the equation
\begin{equation}\label{TPhi'}
\Gamma'_{\sigma\tau,\gamma}(\Gamma'_{\sigma,\tau}(p'_{\sigma}\otimes_{R'}p'_{\tau})\otimes_{R'} p'_{\gamma})=
T^{\Phi'}_{\sigma,\tau,\gamma}\star\Gamma'_{\sigma,\tau\gamma}(p'_{\sigma}\otimes_{R'}\Gamma'_{\tau,\gamma}(p'_{\tau}\otimes_{R'}p'_{\gamma}))\in P'_{\sigma\tau\gamma},
\end{equation}
where ${p}'_{\sigma}\in{P}'_{\sigma},{p}'_{\tau}\in{P}'_{\tau}$ and ${p}'_{\gamma}\in{P}'_{\gamma}.$
Apply the l.h.s. of \eqref{TPhi'} on $p'_{\sigma}\otimes_{R'} p'_{\tau}\otimes_{\small{R}'} p'_{\gamma}\in P'_{\sigma}\otimes_{R'} P'_{\tau}\otimes_{R'} P'_{\gamma}$ using
equations \eqref{TPhi}, \eqref{biject} and \eqref{Gamma'}, keeping in mind that $\psi:R\to R'$ fixes elements of $K$.
$$\begin{array}{cl}
\Gamma'_{\sigma\tau,\gamma}(\Gamma'_{\sigma,\tau}(p'_{\sigma}\otimes_{R'} p'_{\tau})\otimes_{R'} p'_{\gamma})&=
\varphi_{\sigma\tau\gamma}(\Gamma_{\sigma\tau,\gamma}(\varphi^{-1}_{\sigma\tau}(\Gamma'_{\sigma,\tau}(p'_{\sigma}\otimes_{R'} p'_{\tau}))\otimes_{R} \varphi^{-1}_{\gamma}(p'_{\gamma})))\\
&=\varphi_{\sigma\tau\gamma}(\Gamma_{\sigma\tau,\gamma}(\Gamma_{\sigma,\tau}(\varphi^{-1}_{\sigma}(p'_{\sigma})\otimes_R\varphi^{-1}_{\tau}(p'_{\tau}))\otimes_R \varphi^{-1}_{\gamma}(p'_{\gamma})))\\
&=\varphi_{\sigma\tau\gamma}(T^{\Phi}_{\sigma,\tau,\gamma}\star\Gamma_{\sigma,\tau\gamma}
(\varphi^{-1}_{\sigma}(p'_{\sigma})\otimes_R\Gamma_{\tau,\gamma}(\varphi^{-1}_{\tau}(p'_{\tau})\otimes_R \varphi^{-1}_{\gamma}(p'_{\gamma}))))\\
&=\psi(T^{\Phi}_{\sigma,\tau,\gamma})\star\varphi_{\sigma\tau\gamma}(\Gamma_{\sigma,\tau\gamma}
(\varphi^{-1}_{\sigma}(p'_{\sigma})\otimes_R\Gamma_{\tau,\gamma}(\varphi^{-1}_{\tau}(p'_{\tau})\otimes_R \varphi^{-1}_{\gamma}(p'_{\gamma}))))\\
&=T^{\Phi}_{\sigma,\tau,\gamma}\star\varphi_{\sigma\tau\gamma}(\Gamma_{\sigma,\tau\gamma}
(\varphi^{-1}_{\sigma}(p'_{\sigma})\otimes_R\varphi^{-1}_{\tau\gamma}(\Gamma'_{\tau,\gamma}(p'_{\tau}\otimes_{R'} p'_{\gamma}))))\\&=
T^{\Phi}_{\sigma,\tau,\gamma}\star\Gamma'_{\sigma,\tau\gamma}
(p'_{\sigma}\otimes_{R'} \Gamma'_{\tau,\gamma}(p'_{\tau}\otimes_{R'} p'_{\gamma})),
\end{array}$$
which should be equal to the r.h.s. of \eqref{TPhi'}, and so
$T^{\Phi}_{\sigma,\tau,\gamma}=T^{\Phi'}_{\sigma,\tau,\gamma}$ for every $\sigma,\tau,\gamma\in G$.
Consequently, the map $T_{\phi}$
does not depend on the GCC's in their equivariance classes $[\Phi]_{\e}\in\CK$.
This proves that $T_{\phi}$ is well-defined.

(b) We next show that $T_{\phi}$ is multiplicative.
Let $[R],[R']\in\mathcal{C}(K)$ (not necessarily the same classes), and let $\Phi\in$Hom$(G,\p(R))$ and $\Phi'\in$Hom$(G,\p(R'))$ be two GCC's such that $[\Phi]_{\e},[\Phi']_{\e}\in\CK$.
Choose $\{\Gamma_{\sigma,\tau}\}_{\sigma,\tau\in G}$, $\{P_{\sigma}\}_{\sigma\in G}$, $\{\Gamma'_{\sigma,\tau}\}_{\sigma,\tau\in G}$ and $\{P'_{\sigma}\}_{\sigma\in G}$ respectively as above and
denote, for short,
$$\tilde{\Phi}:=\Phi\otimes_{\phi}\Phi',\ \ \tilde{R}:=R\otimes_KR',\ \  \tilde{P}_{\sigma}:={P}_{\sigma}\otimes_{\phi(\sigma)}P'_{\sigma},\ \ \forall \sigma\in G.$$
Then it is not hard to check that
$$\tilde{\Gamma}_{\sigma,\tau}:\begin{array}{ccc}
\tilde{P}_{\sigma}\otimes_{\tilde{R}} \tilde{P}_{\tau}&\to &\tilde{P}_{\sigma\tau}\\
(p_{\sigma}\otimes_{\phi(\sigma)}p'_{\sigma})\otimes_{\tilde{R}}(p_{\tau}\otimes_{\phi(\tau)}p'_{\tau})&\mapsto&
\Gamma_{\sigma,\tau}(p_{\sigma}\otimes_{R}p_{\tau})\otimes_{\phi(\sigma\tau)}\Gamma'_{\sigma,\tau}(p'_{\sigma}\otimes_{R'}p'_{\tau})
\end{array}$$
are $\tilde{R}$-bimodule isomorphisms for every $\sigma,\tau\in G$.
Note that by \eqref{otimeseta} we can identify $R\otimes_{\phi(e)}R'=R\otimes_KR'=\tilde{R}$,
hence the demands \eqref{twoeqs} on $\tilde{\Gamma}_{\sigma,e}$ and $\tilde{\Gamma}_{e,\sigma}$ are satisfied.
As in \eqref{TPhi}, $T_{\sigma,\tau,\gamma}^{\tilde{\Phi}}\in K^*$ is defined by the equation
\begin{equation}\label{TPhitild}
\tilde{\Gamma}_{\sigma\tau,\gamma}(\tilde{\Gamma}_{\sigma,\tau}(\tilde{p}_{\sigma}\otimes_{\tilde{R}}\tilde{p}_{\tau})\otimes_{\tilde{R}} \tilde{p}_{\gamma})=
T^{\tilde{\Phi}}_{\sigma,\tau,\gamma}\star\tilde{\Gamma}_{\sigma,\tau\gamma}(\tilde{p}_{\sigma}\otimes_{\tilde{R}}\tilde{\Gamma}_{\tau,\gamma}
(\tilde{p}_{\tau}\otimes_{\tilde{R}}\tilde{p}_{\gamma}))\in \tilde{P}_{\sigma\tau\gamma},
\end{equation}
where $\tilde{p}_{\sigma}\in\tilde{P}_{\sigma},\tilde{p}_{\tau}\in\tilde{P}_{\tau}$ and $\tilde{p}_{\gamma}\in\tilde{P}_{\gamma}.$
Apply the l.h.s. of \eqref{TPhitild} on
\begin{equation}\label{appled}(p_{\sigma}\otimes_{\phi(\sigma)}p'_{\sigma})\otimes_{\tilde{R}} (p_{\tau}\otimes_{\phi(\tau)}p'_{\tau})\otimes_{\tilde{R}} (p_{\gamma}\otimes_{\phi(\gamma)}p'_{\gamma})
\in \tilde{P}_{\sigma}\otimes_{\tilde{R}} \tilde{P}_{\tau}\otimes_{\tilde{R}} \tilde{P}_{\gamma}.\end{equation}
$$\begin{array}{l}
\tilde{\Gamma}_{\sigma\tau,\gamma}(\tilde{\Gamma}_{\sigma,\tau}((p_{\sigma}\otimes_{\phi(\sigma)}p'_{\sigma})\otimes_{\tilde{R}} (p_{\tau}\otimes_{\phi(\tau)}p'_{\tau}))\otimes_{\tilde{R}}
(p_{\gamma}\otimes_{\phi(\gamma)}p'_{\gamma}))=\\
\tilde{\Gamma}_{\sigma\tau,\gamma}((\Gamma_{\sigma,\tau}(p_{\sigma}\otimes_{R}p_{\tau})\otimes_{\phi(\sigma\tau)}\Gamma'_{\sigma,\tau}(p'_{\sigma}\otimes_{R'}p'_{\tau}))\otimes_{\tilde{R}}
(p_{\gamma}\otimes_{\phi(\gamma)}p'_{\gamma}))=\\
\Gamma_{\sigma\tau,\gamma}(\Gamma_{\sigma,\tau}(p_{\sigma}\otimes_{R}p_{\tau})\otimes_{R}p_{\gamma})\otimes_{\phi(\sigma\tau\gamma)}
\Gamma'_{\sigma\tau,\gamma}(\Gamma'_{\sigma,\tau}(p'_{\sigma}\otimes_{R'}p'_{\tau})\otimes_{R'}p'_{\gamma})=\cdots
\end{array}$$
At this point we can plug the obstruction cocycles $T^{\Phi}_{\sigma,\tau,\gamma},T^{\Phi'}_{\sigma,\tau,\gamma}$ as follows
$$\begin{array}{l}
\cdots=T^{\Phi}_{\sigma,\tau,\gamma}\star\Gamma_{\sigma,\tau\gamma}(p_{\sigma}\otimes_R\Gamma_{\tau,\gamma}(p_{\tau}\otimes_Rp_{\gamma}))
\otimes_{\phi(\sigma\tau\gamma)}T^{{\Phi}'}_{\sigma,\tau,\gamma}\star{\Gamma}'_{\sigma,\tau\gamma}({p}'_{\sigma}\otimes_{{R}'}{\Gamma}'_{\tau,\gamma}
({p}'_{\tau}\otimes_{{R}'}{p}'_{\gamma}))\cdots
\end{array}$$
Equation \eqref{actwosidel} tells us how to pull out scalars
$$\begin{array}{cl}
\cdots&=(T^{\Phi}_{\sigma,\tau,\gamma}\otimes_{K}T^{{\Phi}'}_{\sigma,\tau,\gamma})\star(\Gamma_{\sigma,\tau\gamma}(p_{\sigma}\otimes_R\Gamma_{\tau,\gamma}(p_{\tau}\otimes_Rp_{\gamma}))
\otimes_{\phi(\sigma\tau\gamma)}{\Gamma}'_{\sigma,\tau\gamma}({p}'_{\sigma}\otimes_{{R}'}{\Gamma}'_{\tau,\gamma}
({p}'_{\tau}\otimes_{{R}'}{p}'_{\gamma})))\\
&=(T^{\Phi}_{\sigma,\tau,\gamma}\otimes_{K}T^{{\Phi}'}_{\sigma,\tau,\gamma})\star\tilde{\Gamma}_{\sigma,\tau\gamma}((p_{\sigma}\otimes_{\phi(\sigma)}{p'}_{\sigma})\otimes_{\tilde{R}}
(\Gamma_{\tau,\gamma}(p_{\tau}\otimes_Rp_{\gamma})\otimes_{\phi(\tau\gamma)}{\Gamma}'_{\tau,\gamma}({p}'_{\tau}\otimes_{{R}'}{p'}_{\gamma})))\\
&=(T^{\Phi}_{\sigma,\tau,\gamma}\otimes_{K}T^{{\Phi}'}_{\sigma,\tau,\gamma})\star\tilde{\Gamma}_{\sigma,\tau\gamma}((p_{\sigma}\otimes_{\phi(\sigma)}{p'}_{\sigma})\otimes_{\tilde{R}}
\tilde{\Gamma}_{\tau,\gamma}((p_{\tau}\otimes_{\phi(\tau)}{p}'_{\tau})\otimes_{\tilde{R}}(p_{\gamma}\otimes_{\phi(\gamma)}{p'}_{\gamma}))),
\end{array}$$
which should be equal to the r.h.s. of \eqref{TPhitild} applied on \eqref{appled}. Consequently,
$$T^{\tilde{\Phi}}_{\sigma,\tau,\gamma}=T^{\Phi}_{\sigma,\tau,\gamma}\otimes_{K}T^{{\Phi}'}_{\sigma,\tau,\gamma}, \ \ \forall \sigma,\tau,\gamma\in G.$$
Identifying $K\otimes_{K}K$ with $K$ we obtain that
\begin{equation}\label{mult}T(\Phi\otimes_{\phi}\Phi')=T(\tilde{\Phi})=T(\Phi)\cdot T(\Phi')\in H^3(G,K^*).\end{equation}
By \eqref{mult} we obtain $T_{\phi}([\Phi]_{\e}\circ[\Phi']_{\e})=T_{\phi}([\Phi]_{\e})\cdot T_{\phi}([\Phi']_{\e}),$
hence $T_{\phi}$ is a monoid homomorphism.
\end{proof}
The following is another way to write the fact that $T_{\phi}$ is well-defined on GCC representatives. It is recorded for a later use.
\begin{corollary}\label{equiv2}
Let $[\Phi]_{\e}\in\CK$.
Then $[\Phi]_{\e}\in$Ker$(T_{\phi})$ if and only if there exist $[R]\in\mathcal{C}(K)$ and a GCC $\Phi\in$Hom$(G,\p(R))$ in the equivariance class $[\Phi]_{\e}$
such that $T(\Phi)=0\in H^3(G,K^*_{\phi})$.
\end{corollary}
%\subsection{Surjectivity of the obstruction homomorphism}
The obstruction theory for group extensions with non-abelian kernel is similar to the theory of Clifford system extensions. It was introduced much earlier by S. Eilenberg and S. MacLane in \cite{EM47}.
The group-theoretical obstruction and the homomorphism \eqref{T-phi} are strongly related. Note that any outer $G$-action $\Phi\in\Hom(G,\Out(R))$ on a ring $R$ naturally gives rise to an
outer $G$-action $\Phi'\in\Hom(G,\Out(R^*))$ on the group of $R$-units. Then
\begin{lemma}\label{connect}(see \cite[Theorem 15.2]{EM})
With the notation of \cite[\S 7]{EM47},
$$T(\Phi)=F_3(R^*,\Phi')\in H^3(G,\mathcal{Z}(R)^*).$$
\end{lemma}
\begin{proof}
By direct computation.
\end{proof}

The following result is analogous to \cite[Lemma 9.1]{EM47} for integral domains.
\begin{theorem}\label{surjobs}
Let $[\alpha]\in H^3(G,K^*_{\phi})$, where $\phi\in\text{Hom}(G,\au)$ is any action of a group $G$ on an integral domain $K$.
Then there exist $[R]\in\mathcal{C}(K)$ and a collective character $\Phi\in$Hom$(G,\text{Out}_k(R))$ of type $\phi$ such that $T(\Phi)=[\alpha]$.
Consequently, the homomorphism $\CO\stackrel{T_{\phi}}\to H^3(G,K^*_{\phi})$ is surjective.
\end{theorem}
\begin{proof}
In \cite[\S 9]{EM47}, Eilenberg and MacLane construct the following group, which is denoted here by $H$.
If $|G|>2$ then $H$ is the free (non-abelian) group on the set of generators $(G\setminus\{e\})^2$.
If $G=\Z/2$ then $H$ is the semidirect product of the infinite cyclic group $\langle w\rangle$ acting on the free abelian group $\mathbb{T}^{\Z}$ on the generators $\{p_i\}_{i\in\Z}$ by
$$w(\prod_{i}p_i^{j_i}):=\prod_ip_{i+2}^{j_i},\ \ \prod_{i}p_i^{j_i}\in\mathbb{T}^{\Z}.$$
Clearly, in both cases $H$ has a trivial center, and hence the center of $K^*\times H$ is exactly $K^*$.
We claim that the group $H$ is ordered. For the case $G\neq\Z/2$ the claim follows from \cite[Lemma 3]{V49} for free groups.
In case $G=\Z/2$, we first define for every nontrivial element in $\mathbb{T}^{\Z}$
$$m(\prod_{i}p_i^{j_i}):=j_n, \text{ where }n:=\max\{i|j_i\neq 0\}.$$
Next, order the group $\mathbb{T}^{\Z}$ by naturally defining the positive cone
$$(\mathbb{T}^{\Z})^+:=\left\{ e\neq x\in\mathbb{T}^{\Z}\ |\ \ m(x)>0\right\},$$
and setting
$$x>y\ \  \text{    if     }\ \  xy^{-1}\in(\mathbb{T}^{\Z})^+,\ \ x,y\in\mathbb{T}^{\Z}.$$
Then for every $x>y$ in $\mathbb{T}^{\Z}$ and every $i\in\Z$, we have $w^i(x)>w^i(y)$. We can hence order $H$ by
$$xw^j>yw^i\text{ if } \{j>i\} \text{ or } \{ i=j \text{  and  } x>y\},\ \ xw^j,yw^i\in H.$$

Now, let $R:=KH$ be the group algebra of the ordered group $H$ over the integral domain $K$.
Since $H$ has a trivial center we deduce that $R$ is $K$-central. Furthermore,
by \cite[Lemmas 1.7 and 1.9 ii]{P77}, $R$ has only trivial units, that is $R^*=K^*\times H$.
This observation evidently yields the following claim.
\begin{lemma}\label{extres}
With the above notation, let $\eta\in\Aut(K)$ and $\psi\in \Aut(R^*)$ be automorphisms of the integral domain $K$ and of the group of units of $R=KH$ respectively.
Suppose that $\eta$ and $\psi$ agree on the group of units $K^*$ of $K$. Then
$$\eta\diamond\psi:
\begin{array}{rcl}
R&\to& R\\
\sum_i\alpha_i\cdot h_i&\mapsto&\sum_i\eta(\alpha_i)\cdot\psi(h_i)\end{array}$$
determines a well-defined automorphism of $R$ which extends both $\eta$ and $\psi$.
\end{lemma}
Returning to the proof of the theorem, let $\phi'\in\text{Hom}(G,\Aut (K^*))$ denote the restriction of the $G$-action $\phi$ to the units of $K$.
The group $H$ comes here to fruition.
By \cite[\S 9]{EM47} there exists an outer action
$\Phi'\in$Hom$(G,$Out$(K^*\times H))$ which restricts to $\phi'$ on the center $K^*$ of the group $K^*\times H$,
such that the group theoretical obstruction
\begin{equation}\label{F3R}
F_3(R^*,\Phi')=[\alpha].\end{equation}
%any $\eta\in\Aut$enables the following extension of ${\Phi'}$ to $KH$.
For every $g\in G$ let $\eta'_g\in\Aut(R^*) $ be a representative such that $\eta'_g\cdot\Inn(R)=\Phi'(g)$. Then by Lemma \ref{extres},
$$\phi(g)\diamond\eta'_g\in\Aut_k(R)$$
is an automorphism of $R$ which extends both $\phi(g)$ and $\eta'_g$.
Finally,
$$\Phi:
\begin{array}{rcl}
G&\to& \text{Out}_k(R)\\
g&\mapsto&(\phi(g)\diamond\eta'_g)\cdot\Inn(R)\end{array}$$
is an outer action with $h_R(\Phi)=\phi$ which, by Lemma \ref{connect} and \eqref{F3R}, satisfies $$T(\Phi)=F_3(R^*,\Phi')=[\alpha].$$
\end{proof}
\section{The realization-obstruction sequences}\label{ros}
We can now prove our main theorem.
\begin{proof}
The maps $\Sigma_{\phi},\chi_{\phi}$ and $T_{\phi}$ are homomorphisms of monoids due to Lemmas \ref{units}, \ref{chiphilemma} and \ref{TPwdmorphism} respectively.
Lemmas \ref{units} and \ref{kerim} are responsible for the exactness of \eqref{introexact} in the terms $H^2(G,K^*_{\phi})$ and $\text{Cliff}_k(\phi)$ respectively.
Exactness in $\CK$ is essentially Theorem \ref{CGthm}(3); in order to adjust this result to equivariant classes, combine it with Lemma \ref{equiv1} and Corollary \ref{equiv2}.
The second part of the theorem follows from the partition \eqref{partition2}.
Finally, Theorem \ref{surjobs} yields the result about integral domains. \end{proof}

The sub-monoid \eqref{submon} of $G$-graded Clifford system extensions of type $\phi$ of $K$ itself
appears in an exact sequence of sub-monoids of \eqref{introexact} using equation \eqref{sub3}.
\begin{equation}\label{subexact}
0\to H^2(G,K^*_{\phi})\stackrel{\Sigma_{\phi}}\to\text{Cliff}_k(\phi)\cap\text{Cliff}_k(G,K)\stackrel{\chi_{\phi}}\to\CPK\stackrel{T_{\phi}}\to H^3(G,K^*_{\phi}).
\end{equation}
Notice that when $\phi$ is trivial, then equation \eqref{subexact} describes the central $G$-graded Clifford system extensions of the commutative ring $K$ (compare with \cite[(39)]{CG03}).

Another exact sequence of sub-monoids of \eqref{introexact} represents the obstruction to realizing collective characters by crossed products using \eqref{EXCO}
\begin{equation}\label{exactcp}
0\to H^2(G,K^*_{\phi})\stackrel{\Sigma_{\phi}}\to\EX\stackrel{\chi_{\phi}}\to\CO\stackrel{T_{\phi}}\to H^3(G,K^*_{\phi}).
\end{equation}
Similarly to the above, when $K$ is an integral domain then by Theorem \ref{surjobs} the sequence \eqref{exactcp} terminates, that is
$$0\to H^2(G,K^*_{\phi})\stackrel{\Sigma_{\phi}}\to\EX\stackrel{\chi_{\phi}}\to\CO\stackrel{T_{\phi}}\to H^3(G,K^*_{\phi})\to 0.$$

\section{Central simple algebras}\label{Csa}
The objective of this section is to present the well-known Brauer theoretical restriction-obstruction exact group sequence (see \eqref{Brseq})
as an image of an exact sequence \eqref{exactcpcs} of sub-monoids of the above sequence \eqref{exactcp}. This implementation is given in Theorem \ref{commute}.
We first need the following

\begin{lemma}\label{etaR1}
Let $[P]\in$Pic$(R)$, where $R$ is a simple ring.
Then there exists $\eta\in\text{Aut}(R)$ such that $P$ is isomorphic to $R_{1,\eta}$ as an $R$-bimodule.
\end{lemma}
\begin{proof}
Since $R$ is simple, invertibility of $P$ as an $R$-bimodule implies that it is isomorphic to a free $R$-module of rank 1 as left and right module. %(perhaps in two different ways).
By a suitable isomorphism of $R$-bimodules we can identify $P$ with $R$ itself, acted from the left by multiplication, and from the right using the notation $``\cdot"$. Under this identification,
define $\eta:\begin{array}{ccc}R&\to &R\\ r&\mapsto& 1\cdot r\end{array}.$
For every $r_1,r_2\in R$ we have
$$\eta(r_1r_2)=1\cdot(r_1r_2)=(1\cdot r_1)\cdot r_2=\eta(r_1)\cdot r_2=(\eta(r_1)\star 1)\cdot r_2=\eta(r_1) (1\cdot r_2)=\eta(r_1)\eta(r_2),$$
and therefore $\eta$ is a ring endomorphism (additivity is clear). Since $R$ is simple then the nonzero endomorphism $\eta$ (notice that $\eta(1)=1$) is in $\text{Aut}(R)$.
Moreover, for every $r,s\in R$
$$ s\cdot r=(s\star 1)\cdot r=s\star (1\cdot r)=s \eta(r).$$
%$$r\star s=r\star (1\cdot s)=(r\star 1)\cdot s=\eta(r_1)\cdot s.$$
By \eqref{star} the $R$-bimodule structure is the same as that of $R_{1,\eta}$.
\end{proof}
As from this point, $K$ will be assumed to be a field. Recall that a
$K$-central simple algebra is a simple algebra $R$ which is
finite-dimensional over $\mathcal{Z}(R)=K$. It is well-known that
the tensor product over $K$ of $K$-central simple algebras is
again $K$-central simple. Consequently, the $K$-isomorphism
classes of $K$-central simple algebras form a sub-monoid
$\CS<\mathcal{C}(K)$.
\begin{lemma}\label{cslemma}
Let $[R]\in\CS$. Then
\begin{enumerate}
\item For every subalgebra $K'\subseteq K$, the embedding $\iota_R:$Out$_{K'}(R)\hookrightarrow$Pic$_{K'}(R)$ (see \eqref{out}) is surjective.
\item The morphism $h_R:\p(R)\to \au$ (see \eqref{hR}) is injective.
\item (see \cite[Theorem 3.3.1]{NVO04} for a more general setup) Strongly graded algebras admitting a $K$-central simple base algebra are crossed products.
\end{enumerate}
\begin{proof}
Let $[P]\in$Pic$_{K'}(R)$.
By Lemma \ref{etaR1}, $P$ is isomorphic as an $R$-bimodule to $R_{1,\eta}$ for some $\eta\in\text{Aut}(R)$.
Since the left and right $K'$-actions on $P$ agree, then  $\eta\in\text{Aut}_{K'}(R)$, and so
$$[P]=[R_{1,\eta}]=\iota_R(\eta\cdot\text{Inn}(R))\in\iota_R(\text{Out}_{K'}(K))$$ proving (1).
In particular, putting $K'=K$, we obtain that Pic$_{K}(R)\cong$Out$_{K}(R)$, which is trivial by the Noether-Skolem Theorem. Exactness of the sequence \eqref{bassequence} establishes (2).
Finally, by Lemma \ref{cslemma}(1), GCC's over $K$-central simple base algebras are just collective characters. Applying Lemma \ref{CPOut} we deduce (3).
\end{proof}
\end{lemma}
As before, let $\phi\in\text{Hom}(G,\au)$, where $G$ is any group. Since $\CS$ is closed under the $\mathcal{C}(K)$-product, there is an exact sequence of sub-monoids of \eqref{exactcp}
\begin{equation}\label{exactcpcs}
0\to H^2(G,K^*_{\phi})\stackrel{\Sigma_{\phi}}\to\EXCS\stackrel{\chi_{\phi}}\to\COCS\stackrel{T_{\phi}}\to H^3(G,K^*_{\phi}),
\end{equation}
where
\begin{equation}\label{excs}
\EXCS:=\left\{[R*G]\in\EX|\ \ [R]\in\CS\right\}
\end{equation} and
\begin{equation}\label{COCS}
\COCS:=\left\{[\Phi]_{\e}\in\CO |\ \ \Phi\in\text{Hom}(G,\text{Out}_k (R)),[R]\in\CS\right\}_{}.\end{equation}
The term \eqref{COCS} in the sequence \eqref{exactcpcs} can be interpreted using the generalized notion of normality (see Definition \ref{normal}) as follows.
Let $$\CSP:=\mathcal{C}^{\phi}(K)\cap\CS$$ be the set of isomorphism classes of $\phi$-normal $K$-central simple algebras.
As an intersection of two sub-monoids, $\CSP$ is itself a sub-monoid of $\mathcal{C}(K)$.
\begin{lemma}
The morphism \eqref{Psiphi} determines an isomorphism of abelian monoids
\begin{equation}\label{isonorm}\Psi_{\phi}:\COCS\xrightarrow{\cong}\CSP.\end{equation}
\end{lemma}
\begin{proof}
Clearly, $\Psi_{\phi}(\COCS)\subseteq\CS.$
Additionally, by Lemma \ref{resnormal}, $$\Psi_{\phi}(\COCS)\subseteq\Psi_{\phi}(\CO)\subseteq\mathcal{C}^{\phi}(K),$$ and so by definition $\Psi_{\phi}(\COCS)\subseteq\CSP.$
The rest of the proof shows the bijectivity of \eqref{isonorm}.

\textbf{Injectivity of \eqref{isonorm}.} Let
$\Phi\in\text{Hom}(G,\text{Out}_k(R))$ and $\Phi'\in\text{Hom}(G,\text{Out}_k(R'))$ be collective characters such that $[\Phi]_{\e},[\Phi']_{\e}\in\COCS$, that is
\begin{equation}\label{hh}
h([\Phi]_{\e})=h([\Phi']_{\e})=\phi.
\end{equation}
Assume. additionally that
\begin{equation}\label{eq}\Psi_{\phi}([\Phi]_{\e})=\Psi_{\phi}([\Phi']_{\e})\in\CSP.\end{equation}
We need to show that these two collective characters are equivariant. Equation \eqref{eq} says that there exists a $K$-algebra isomorphism $\psi:R\to R'$. We claim that the outer automorphisms
$\Phi(g)$ and $\Phi'(g)$ are $\psi$-equivariant for every $g\in G$ (see Definition \ref{equiGCC}).
Indeed, let
$$\eta=\eta(g)\in\text{Aut}_k(R),\ \ \eta'=\eta'(g)\in\text{Aut}_k(R'),\ \ g\in G$$ such that
$\eta\cdot$Inn$(R)=\Phi(g)$ and $\eta'\cdot$Inn$(R')=\Phi'(g)$. By \eqref{hh}, both automorphisms restrict to $\phi$ over their center, i.e.
\begin{equation}\label{etaeta}
\eta|_K=\eta'|_K=\phi.
\end{equation}
Now, by \eqref{etaeta} the rule $$r\mapsto\psi^{-1}\eta'^{-1}\psi\eta(r),\ \ r\in R$$ determines a $K$-automorphism of $R$, and is thus inner owing to the Noether-Skolem Theorem.
Consequently, there exists a unit $x\in R^*$ such that
\begin{equation}\label{uptoinn}
\psi\eta\iota_x=\eta'\psi\in \text{Hom}_K(R,R'),
\end{equation}
where $\iota_x\in\text{Inn}(R)$ denotes the conjugation by $x$ in $R$. Bearing in mind that $\eta\iota_x\cdot$Inn$(R)=\eta\cdot$Inn$(R)=\Phi(g)$, define
\begin{eqnarray}\label{varphi}\varphi:\begin{array}{ccc}R_{1,\eta\iota_x}&\to& R'_{1,\eta'}\\s&\mapsto &\psi(s)\end{array},\end{eqnarray}
thinking of $s\in R_{1,\eta\iota_x}$ as lying in $R$. We show that \eqref{varphi} is a $\psi$-equivariance map. Let $r\in R$ and $s\in R_{1,\eta\iota_x}$.
Firstly, for left action
$$\varphi(r\star s)=\varphi(rs)=\psi(rs)=\psi(r)\psi(s)=\psi(r)\star\varphi(s).$$
Next, for right action we use \eqref{star}
$$\varphi(s\cdot r)=\varphi(s\eta\iota_x(r))=\psi(s\eta\iota_x(r))=\psi(s)\psi(\eta\iota_x(r))=\varphi(s)\psi\eta\iota_x(r)=\cdots$$
Apply \eqref{uptoinn} and \eqref{star} again to obtain
$$\cdots=\varphi(s)\eta'(\psi(r))=\varphi(s)\cdot\psi(r),$$
and so $\varphi$ satisfies the equivariance condition (Definition \ref{defeq}).
This says that $R_{1,\eta\iota_x}$ and $R'_{1,\eta'}$ are $\psi$-equivariant. Since this $\psi$-equivariance holds for every $g\in G$, we deduce that $[\Phi]_{\e}=[\Phi']_{\e}$.
Therefore \eqref{isonorm} is injective.

\textbf{Surjectivity of \eqref{isonorm}.} We need to find a pre-image under $\Psi_{\phi}$ for any $\phi$-normal $K$-central simple algebra $R$.
Indeed, by the $\phi$-normality property of $R$, the $K$-automorphisms $\phi(g)$ can be extended to $R$-automorphisms $\eta(g)$ for every $g\in G$.
By definition, $\phi(g)$ lies in the image of \eqref{lift} for every $g\in G$. Let $\eta(g)$ be any pre-image of $\phi(g)$ under \eqref{lift}, and let
\begin{eqnarray}\label{PhiRphi}\Phi:=\begin{array}{ccl}
G&\to &\text{Out}_k(R)\\
g&\mapsto&\eta(g)\cdot\text{Inn}(R).\end{array}\end{eqnarray}
Then $\Phi$ extends $\phi$, that is
\begin{equation}\label{extends}
\phi=h_R\circ\Phi.\end{equation}
Note that existence of a map \eqref{PhiRphi} which extends $\phi$ stems from the normality property of $R$, independently of the simplicity of this $K$-central algebra.
However, in general \eqref{PhiRphi} is not a collective character. When the $\phi$-normal $K$-central algebra $R$ is simple, then the group morphism $h_R$ is injective (Lemma \ref{cslemma}(2)).
Together with \eqref{extends},
and the fact that $\phi$ is a group morphism as well we obtain that $\Phi\in$Hom$(G,\text{Out}_k(R))$.
Hence $[\Phi]_{\e}\in\COCS$, and we found a pre-image of $[R]\in\CSP$ under the morphism $\Psi_{\phi}$.
\end{proof}
\begin{notation}
We denote the Brauer similarity relation in $\CS$ (see \cite[Page 93]{H}) by $\sim_{\Br}$, and write the Brauer group of $K$ as
$$\BrK=\bigslant{\CS}{\sim_{\Br}}.$$\end{notation}
A $K$-central simple algebra which is Brauer similar to a $\phi$-normal algebra is itself $\phi$-normal \cite[Theorem 5.4]{EM}.
There is a map $T:\text{Br}^{\phi}(K)\to H^3(G,K^*_{\phi})$ from the subgroup Br$^{\phi}(K)<$Br$(K)$ of $\phi$-normal similarity classes in the Brauer group of $K$
(see \cite[\S 6]{EM} for a Galois action $\phi$),
which is essentially the map $T_{\phi}$. Here is a more precise formulation without a proof.
\begin{lemma}\label{noprf}
Let $R$ be a $\phi$-normal $K$-central simple algebra. Then with the above notation
$$T([R]_{\text{Br}})=T_{\phi}(\Psi_{\phi}^{-1}[R])\in H^3(G,K^*_{\phi}).$$
\end{lemma}
From here onwards, suppose that $G$ is finite, and that $\phi\in\text{Hom}(G,\au)$ is a Galois action, that is $\phi$ is {faithful} with $k=K^{\phi}$ its fixed field.
Recall the role of crossed products in the Galois setting (see \cite[\S 4.4]{H}).
Firstly, the corresponding skew group algebra $K_{\phi}G$ is isomorphic to an algebra of matrices over $k$. Furthermore, the subgroup of $\Brk$
of Brauer similarity classes of central $k$-simple algebras that are split by $K$ is isomorphic to $H^2(G,K^*_{\phi})$ and its elements are represented by the crossed products $K^{\alpha}_{\phi}G$,
where $[\alpha]$
runs over the cohomology classes in $H^2(G,K^*_{\phi})$. This is summarized in the following exact sequence
\begin{equation}\label{Brseq}
0\to H^2(G,K^*_{\phi})\to\text{Br}(k)\stackrel{\res^k_K}\to\text{Br}^{\phi}(K)\stackrel{T}\to H^3(G,K^*_{\phi}),
\end{equation}
where the restriction map
$$\res^k_K:\begin{array}{rcl}\text{Br}(k)&\to&\text{Br}(K)\\
~[A]_{\Brk}&\mapsto&[A\otimes_kK]_{\BrK}\end{array}$$
has its image of lying in the subgroup Br$^{\phi}(K)<$Br$(K)$ of $\phi$-normal classes.

Next, Lemma \ref{lem1} herein interprets the term \eqref{excs} in the sequence \eqref{exactcpcs} in the Galois setting.

\begin{lemma}\label{lem1}
Let $\phi\in\text{Hom}(G,\au)$ be a Galois action, and let $[R*G]\in\EXCS$. Then
\begin{enumerate}
\item The crossed product $R*G$ is $k$-central simple.
\item The map
\begin{equation}\label{rho}\begin{array}{rcl}\EXCS&\to&\Brk\\
~[R*G]&\mapsto&[R*G]_{\Brk}
\end{array}\end{equation}
is a well-defined surjective homomorphism of monoids.
\item There is a $K$-algebra isomorphism $R*G\otimes_kK\cong\text{End}_R(R*G)$. In particular,
\begin{equation}\label{res}
\res^k_K([R*G]_{\Brk})=[R]_{\BrK}.\end{equation}
\end{enumerate}
\end{lemma}
\begin{proof}
Simplicity of $R*G=\oplus_{g\in G}Ru_g$ is proven using a standard argument saying that if $x$ is a nonzero element of shortest length in a two-sided $R*G$-ideal $I$,
then by faithfulness of the $G$-action on $K=\mathcal{Z}(R)$, $x$ is of the form $ru_g$ for some $r\in R$ and $g\in G$.
Since $R$ is itself simple it follows that $I=R*G$.
The faithfulness property of $\phi$ also ascertains that a central element $z\in \mathcal{Z}(R*G)$ is in $Ru_e$. Since $z$ centralizes $R$ it is actually in $K$, and
since it is invariant under $\phi$ we deduce that $z\in k$. Clearly, $k$ lies in the center of $R*G$, and so $\mathcal{Z}(R*G)=k$ proving (1).

The fact that \eqref{rho} is well-defined is straightforward. To prove that this map is a homomorphism we need to show that for every $[R*G],[R'*G]\in\EXCS$
\begin{equation}\label{simbr}
R*G\otimes_{\phi}R'*G\stackrel{?}\sim_{\Br}R*G\otimes_{k}R'*G.
\end{equation}
We follow the proof in \cite[Theorem 4.4.3]{H}, originally proven for the case where $R$ and $R'$ are isomorphic to $K$ (that is $[R*G],[R'*G]\in\EXCS\cap\text{Cliff}_k(G,K)$).
Let $R*G=\oplus_{g\in G}Ru_g,\ \  R'*G=\oplus_{g\in G}R'u'_g,$ and write for short $A:=R*G\otimes_{k}R'*G$.
By \cite[Lemma 4.4.3]{H} there exist orthogonal idempotents $\{e_g\}_{g\in G}$ in $K\otimes_{k} K$ such that
$$K\otimes_{k} K=\bigoplus_{g\in G}e_g(K\otimes_{k} 1)=\bigoplus_{g\in G}e_g(1\otimes_{k} K),$$ and for every $g\in G$ and $x\in K$
\begin{equation}\label{eg}
e_g(x\otimes_{k} 1)=e_g(1\otimes_{k} \phi(g)(x)).\end{equation}
Moreover, by \eqref{eg} and the fact that $K\otimes_{k} K$ is central in $R\otimes_{k} R'$ we get the additional properties of $e:=e_1\in K\otimes_{k} K\subseteq A$ (see \cite[pages 115-116]{H}).
For every $g\neq h\in G$
\begin{equation}\label{gneqh}
e\cdot (Ru_g\otimes_{k}R'u'_h)\cdot e=0,\end{equation}
as well as for every $ru_g\otimes_{k}r'u'_g\in Ru_g\otimes_{k}R'u'_g$
\begin{equation}\label{ae}
(ru_g\otimes_{k}r'u'_g)\cdot e=e\cdot (ru_g\otimes_{k}r'u'_g)=e\cdot (ru_g\otimes_{k}r'u'_g)\cdot e.
\end{equation}
Firstly, by \eqref{gneqh}
\begin{equation}\label{eAecong}
e\cdot A\cdot e=e\cdot (\bigoplus_{g,h\in G}Ru_g\otimes_{k}R'u'_h)\cdot e=e\cdot (\bigoplus_{g\in G}Ru_g\otimes_{k}R'u'_g)\cdot e.
\end{equation}
We next claim that
\begin{eqnarray}\label{eAe}\beta:\begin{array}{ccc}
R*G\otimes_{\phi}R'*G &\to& e\cdot A\cdot e\\
 ru_g\otimes_{\phi(g)}r'u'_g &\mapsto& e\cdot (ru_g\otimes_{k}r'u'_g)\cdot e
\end{array}\end{eqnarray}
determines an isomorphism of $k$-algebras.
Indeed, for every $ru_g\otimes_{\phi(g)}r'u'_g$ and $su_h\otimes_{\phi(h)}s'u'_h$ in $R*G\otimes_{\phi}R'*G$
$$\begin{array}{rl}
\beta((ru_g\otimes_{\phi(g)}r'u'_g)(su_h\otimes_{\phi(h)}s'u'_h))&=\beta(ru_gsu_h\otimes_{\phi(gh)}r'u'_gs'u'_h)\\
&=e\cdot(ru_gsu_h\otimes_{k}r'u'_gs'u'_h)\cdot e\\
&=e\cdot(ru_g\otimes_{k}r'u'_g)\cdot(su_h\otimes_{k}s'u'_h)\cdot e=\cdots\end{array}$$
Apply \eqref{ae} to obtain
$$\cdots=e\cdot(ru_g\otimes_{k}r'u'_g)\cdot e\cdot(su_h\otimes_{k}s'u'_h)\cdot e=\beta(ru_g\otimes_{\phi(g)}r'u'_g)\beta(su_h\otimes_{\phi(h)}s'u'_h),$$
hence, \eqref{eAe} is a $k$-algebra homomorphism. Surjectivity of this homomorphism is a direct consequence of \eqref{eAecong}. Evidently, \eqref{eAe} is injective because its domain is simple.
Consequently, \eqref{eAe} is a $k$-algebra isomorphism.

Finally, by \cite[Sublemma, Page 114]{H}
$$R*G\otimes_{k}R'*G=A\sim_{\Br} e\cdot A\cdot e,$$
and by \eqref{eAe} $e\cdot A\cdot e\cong R*G\otimes_{\phi}R'*G,$
these two facts verify \eqref{simbr}.

To prove that \eqref{rho} is surjective,
we show that any $k$-central simple algebra $A$ is Brauer similar to a crossed product $R*G$ with $[R*G]\in\EXCS$.
Consider the tensor product $A\otimes_kK_{\phi}G$, where $K_{\phi}G=\bigoplus_{g\in G}\text{Span}_K\{u_g\}$ is the skew group algebra with respect to $\phi$.
This tensor product is a $G$-graded $k$-algebra admitting an invertible element $1\otimes_ku_g$ in each homogeneous component.
Hence, it is a crossed product $R*G$ over the $K$-central simple base algebra $$R:=A\otimes_k\text{Span}_K\{u_e\}\cong A\otimes_kK.$$
Moreover, the outer action of $G$ on $R$ is via the conjugation by the elements $1\otimes_ku_g$, therefore its restriction to the center $K$ coincides with $\phi$. Hence, we deduce that
$[R*G]\in\EXCS$.
Since $K_{\phi}G$ is Brauer trivial, we have $$A\sim_{\Br} A\otimes_kK_{\phi}G= R*G,$$ proving the second claim of the lemma.

To prove the last part of the lemma, send
$$t=\sum_{g,i} r^i_gu_g\otimes_kx^i_g\in R*G\otimes_kK$$ to the endomorphism
$$\varphi_t:y\mapsto \sum_{g,i}r^i_gu_g yx^i_g, \ \ y\in R*G.$$
Then it is not hard to check that $\varphi_t(y\cdot r)=\varphi_t(y)\cdot r$ for every $y\in R*G, r\in R$, and that
\begin{eqnarray}\label{resiso}\begin{array}{ccc}
R*G\otimes_kK&\to&\text{End}_R(R*G)\\
t&\mapsto&\varphi_t
\end{array}\end{eqnarray} determines a $K$-algebra homomorphism.
As can easily be verified, \eqref{resiso} takes $t=u_e\otimes_k1\in R*G\otimes_kK$ to the identity $\varphi_t=\Id_{R*G}$.
Since the domain $R*G\otimes_kK$ is simple, the nonzero algebra homomorphism \eqref{resiso} is injective.
Next, the crossed product $R*G$ is free over $R$, and therefore $\text{End}_R(R*G)$ is a full ring of $|G|\times|G|$-matrices over $R$, in particular it is Brauer similar to $R$, that is
\begin{equation}\label{t}
[\text{End}_R(R*G)]_{\BrK}=[R]_{\BrK}.
\end{equation}
Comparing dimensions over $K$ of the domain and the range of \eqref{resiso} we obtain
$$\dim_K(R*G\otimes_kK)=|G|^2\cdot\dim_K(R)=\dim_K(\text{End}_R(R*G)),$$
and thus the injective homomorphism \eqref{resiso} is an isomorphism as required.
The isomorphism \eqref{resiso} together with \eqref{t} verify \eqref{res}.
\end{proof}
We can now construct the following diagram, whose commutativity (Theorem \ref{commute}) manifests the exact sequence \eqref{Brseq} as an image of the exact sequence \eqref{exactcpcs}.

\begin{center}
\begin{tikzpicture}
\tikzset{thick arc/.style={->, black, fill=none,  >=stealth,
text=black}} \tikzset{node distance=1cm, auto}
\node (A) {$0$};
\tikzset{node distance=2cm, auto}
\node (B) [right of=A] {$H^2(G,K^*_{\phi})$};
\tikzset{node distance=3cm, auto}
\node (C) [right of=B] {$\EXCS$};
\node (D) [right of=C] {$\COCS$};
\node (E) [right of=D] {$H^3(G,K^*_{\phi})$    \eqref{exactcpcs} };
\tikzset{node distance=2.4cm, auto}
\node (F) [below of=A] { $0$};
\tikzset{node distance=2cm, auto}
\node (G) [right of=F] {$H^2(G,K^*_{\phi})$};
\tikzset{node distance=3cm, auto}
\node (H) [right of=G] {$\text{Br}(k)$};
\node (I) [right of=H] {$\text{Br}^{\phi}(K)$};
\node (J) [right of=I] {$H^3(G,K^*_{\phi})$    \eqref{Brseq}};
\tikzset{node distance=1.2cm, auto}
\node (K) [below of=D] {$\CSP$};
\node (VEQ1) [below of=B] {$\veq$};v
\node (VEQ1) [below of=A] {$\circledast$};
\node (VEQ2) [below of=E] {$\veq$};
\draw[thick arc,
draw=black] (A) to node [near start]  {}
(B);
\draw[thick arc,
draw=black] (B) to node [near start]  {$\Sigma_{\phi} $}
(C);
\draw[thick arc,
draw=black] (C) to node [near start]  {$\chi_{\phi} $}
(D);
\draw[thick arc,
draw=black] (D) to node [near start]  {$T_{\phi} $}
(E);
\draw[thick arc,
draw=black] (F) to node [near start]  {}
(G);
\draw[thick arc,
draw=black] (G) to node [near start]  {}
(H);
\draw[thick arc,
draw=black] (H) to node [near start]  {$\small{\res^k_K}$}
(I);
\draw[thick arc,
draw=black] (I) to node [near start]  {$T $}
(J);
\draw[<->,
draw=black] (D) to node [near start] [left]  {$\eqref{isonorm} $}
(K);
\draw[->>,
draw=black] (K) to node [near start] [left]  {$\thicksim_K $}
(I);
\draw[->>,
draw=black] (C) to node [left]  {$\eqref{rho} $}
(H);
\end{tikzpicture}
\end{center}
\begin{theorem}\label{commute}
Let $\phi\in\text{Hom}(G,\au)$ be a Galois action. Then the diagram $\circledast$ is commutative.
\end{theorem}
\begin{proof}
Commutativity of the right square is just Lemma \ref{noprf}.
Next, by Lemma \ref{Re} and equation \eqref{res} we obtain that
$$[\Psi_{\phi}\circ\chi_{\phi}[R*G]]_{\BrK}=[R]_{\BrK}=\res^k_K([R*G]_{\Brk})$$
for every $[R*G]\in\EXCS$, proving commutativity of the central square.
Commutativity of the left square is clear.
\end{proof}
\noindent{\bf Acknowledgement.} The author is indebted to E. Aljadeff for suggesting Theorem \ref{surjobs}, and to A. Antony, D. Blanc and O. Schnabel for valuable discussions,
and to the referee for useful comments.

{}
\end{document}